\documentclass[11pt]{amsart}
\usepackage{amssymb}
\usepackage{mathabx}
\usepackage{xypic}
\usepackage{hhline}
\usepackage{dcpic}
\usepackage{hyperref}
\usepackage{tikz-cd}
\usepackage{rotating}

\newtheorem{theorem}{Theorem}[section]
\newtheorem{lemma}[theorem]{Lemma}

\newtheorem{proposition}[theorem]{Proposition}
\newtheorem{corollary}[theorem]{Corollary}

\newtheorem{definition}[theorem]{Definition}

\theoremstyle{remark}
\newtheorem{remark}[theorem]{Remark}

\numberwithin{theorem}{subsection}
\numberwithin{equation}{subsection}

\begin{document}


\makeatletter
\def\Ddots{\mathinner{\mkern1mu\raise\p@
\vbox{\kern7\p@\hbox{.}}\mkern2mu
\raise4\p@\hbox{.}\mkern2mu\raise7\p@\hbox{.}\mkern1mu}}
\makeatother

\newcommand{\OP}[1]{\operatorname{#1}}
\newcommand{\GO}{\OP{GO}}
\newcommand{\leftexp}[2]{{\vphantom{#2}}^{#1}{#2}}
\newcommand{\leftsub}[2]{{\vphantom{#2}}_{#1}{#2}}
\newcommand{\rightexp}[2]{{{#1}}^{#2}}
\newcommand{\rightsub}[2]{{{#1}}_{#2}}
\newcommand{\AI}{\OP{AI}}
\newcommand{\gen}{\OP{gen}}
\newcommand{\triv}{\OP{triv}}
\newcommand{\Asai}{\OP{Asai}}
\newcommand{\Image}{\OP{Im}}
\newcommand{\Res}{\OP{Res}}
\newcommand{\Ad}{\OP{Ad}}
\newcommand{\Out}{\OP{Out}}
\newcommand{\tr}{\OP{tr}}
\newcommand{\spec}{\OP{spec}}
\newcommand{\scopy}{\OP{end}}
\newcommand{\ord}{\OP{ord}}
\newcommand{\Cent}{\OP{Cent}}
\newcommand{\wellip}{\OP{w-ell}}
\newcommand{\Nrd}{\OP{Nrd}}
\newcommand{\Std}{\OP{Std}}
\newcommand{\JL}{\OP{JL}}
\newcommand{\temp}{\OP{temp}}
\newcommand{\BC}{\OP{BC}}
\newcommand{\sgn}{\OP{sgn}}
\newcommand{\SU}{\OP{SU}}
\newcommand{\Hom}{\OP{Hom}}
\newcommand{\Inter}{\OP{Int}}
\newcommand{\diag}{\OP{diag}}
\newcommand{\Sym}{\OP{Sym}}
\newcommand{\GSp}{\OP{GSp}}
\newcommand{\GL}{\OP{GL}}
\newcommand{\GSO}{\OP{GSO}}
\newcommand{\bdd}{\OP{bdd}}
\newcommand{\Int}{\OP{Int}}
\newcommand{\art}{\OP{art}}
\newcommand{\vol}{\OP{vol}}
\newcommand{\cusp}{\OP{cusp}}
\newcommand{\un}{\OP{un}}
\newcommand{\ellip}{\OP{ell}}
\newcommand{\sph}{\OP{sph}}
\newcommand{\gsimp}{\OP{sim-gen}}
\newcommand{\Aut}{\OP{Aut}}
\newcommand{\disc}{\OP{disc}}
\newcommand{\sdisc}{\OP{s-disc}}
\newcommand{\aut}{\OP{aut}}
\newcommand{\End}{\OP{End}}
\newcommand{\barQ}{\OP{\overline{\mathbf{Q}}}}
\newcommand{\barQp}{\OP{\overline{\mathbf{Q}}_{\it p}}}
\newcommand{\Gal}{\OP{Gal}}
\newcommand{\simp}{\OP{sim}}
\newcommand{\odd}{\OP{odd}}
\newcommand{\Normal}{\OP{Norm}}
\newcommand{\Ind}{\OP{Ind}}
\newcommand{\St}{\OP{St}}
\newcommand{\unit}{\OP{unit}}
\newcommand{\reg}{\OP{reg}}
\newcommand{\SL}{\OP{SL}}
\newcommand{\Frob}{\OP{Frob}}
\newcommand{\Id}{\OP{Id}}
\newcommand{\GSpin}{\OP{GSpin}}
\newcommand{\Norm}{\OP{Norm}}
\newcommand{\SO}{\OP{SO}}
\newcommand{\PGSO}{\OP{PGSO}}
\newcommand{\G}{\OP{G}}
\newcommand{\BO}{\OP{O}}
\newcommand{\Spin}{\OP{Spin}}
\newcommand{\PGL}{\OP{PGL}}
\newcommand{\Lie}{\OP{Lie}}
\newcommand{\Der}{\OP{Der}}
\newcommand{\Pin}{\OP{Pin}}
\newcommand{\I}{\OP{I}}
\newcommand{\PGU}{\OP{PGU}}
\newcommand{\std}{\OP{std}}

\title[The triality of the twisted discrete trace formula for $\PGSO(8)$]{ The triality of the twisted discrete trace formula for $\PGSO(8)$}

\author{Tuoping Du, Zhifeng Peng*, Haoyang Wang}
\address{ Tuoping Du, School of Mathematics and Physics, North China  Electric Power University, Beijing, 102206, P.R. China.}
\address{Zhifeng Peng*, SooChow University No.1 Shizi Street, Suzhou City, Jiangsu Province China}
\address{Haoyang Wang, SooChow University No.1 Shizi Street, Suzhou City, Jiangsu Province China}
\email{ dtp1982@163.com, tpdu@ncepu.edu.cn, zfpeng@suda.edu.cn, 20244207002@stu.suda.edu.cn}

\begin{abstract}
In this paper, we establish the triality twisted trace formula for $\PGSO(8)$, including its discrete part, and obtain a coarse classification of its automorphic representations by combining the properties of triality. By comparing the standard trace formula for $G_{2}$ with the triality twisted trace formula for $\PGSO(8)$, we derive a corresponding coarse classification for automorphic representations of $G_{2}$. Specifically, we construct the triality-twisted elliptic endoscopic data for $\PGSO(8)$ and the elliptic endoscopic data for $G_2$. Based on these constructions and the general framework of trace formulas, we establish the relevant trace formulas. Utilizing the triality property of $\PGSO(8)$, we obtain a coarse classification of its automorphic representations, which in turn yields a coarse classification for those of $G_{2}$.

\end{abstract}

\maketitle

\tableofcontents
{\section{\textbf{Introduction}}
The trace formula is a fundamental tool in the study of automorphic representations of connected reductive groups. Arthur \cite{A1} developed the standard trace formula and applied it to classify automorphic representations of classical groups. His approach is based on the theory of endoscopy in both the standard and twisted case, comparing the trace formula for classical groups with the twisted trace formula for general linear groups. This method, however, relies on the assumption that the twisted trace formula for general linear groups is stabilized. In 2014, Moeglin and Waldspurger \cite{MW2,MW3} established the stabilization of the twisted trace formula for reductive groups. Subsequently, in 2015, Mok \cite{M} obtained analogous results for quasi-split unitary groups using Arthur's techniques.

In this work, we apply the triality twisted trace formula for $\mathrm{PGSO}(8)$ to investigate the classification of automorphic representations of $\mathrm{G}_2$, by comparing the trace formula for $\mathrm{G}_2$ with the triality twisted trace formula for $\mathrm{PGSO}(8)$.

The Lie algebra of type $D_4$ exhibits richer symmetry than generic type $D_n$ Lie algebras. In Lie theory, the outer automorphism group $\mathrm{Out}(D_4)$ is isomorphic to the symmetric group $S_3$, whereas for $n \ne 4$,
\[
\mathrm{Out}(D_n) \cong \mathbb{Z}/2\mathbb{Z}.
\]
Hence, the Lie algebra of type $D_4$ admits an outer automorphism $\theta$ of order three, known as \emph{triality}. At the group level, the outer automorphism group of $\SO(8)$ is $\mathbb{Z}/2\mathbb{Z}$, where the nontrivial outer automorphism is realized by an element in $\mathrm{O}(8) \setminus \SO(8)$. Consequently, the triality automorphism $\theta$ of $\mathfrak{so}(8)$ does not lift to $\SO(8)$, but does lift to the simply connected group $\Spin(8)$ and to the adjoint group $\mathrm{PGSO}(8)$. The center of $\Spin(8)$ is the finite group scheme $\mu_2 \times \mu_2$, which contains three nontrivial subgroups of order two. The triality automorphism permutes these three subgroups while preserving the center, and therefore descends to an automorphism of the adjoint group $\mathrm{PGSO}(8)$.

The complex group $\Spin(8)$, which is the dual group of $\mathrm{PGSO}(8)$, admits three non-isomorphic $8$-dimensional irreducible orthogonal representations: the standard representation
\[
\rho: \Spin(8) \to \SO(8),
\]
and two half-spin representations
\[
\rho^{+}: \Spin(8) \to \GL({\textstyle\bigwedge}^{+} W), \quad 
\rho^{-}: \Spin(8) \to \GL({\textstyle\bigwedge}^{-} W),
\]
which are cyclically permuted by the triality automorphism. Then they define an embedding
\[
i = (\rho, \rho^{+}, \rho^{-}): \Spin(8) \hookrightarrow \SO(8)^3.
\]

This embedding is equivariant with respect to an action of $S_3$ on $\Spin(8)$, which permutes the isomorphism classes of the three representations. In fact, this action is induced by the permutation action on $\SO(8)^3$. Correspondingly, there exist three non-conjugate homomorphisms
\[
\lambda_i: \SO(8) \to \mathrm{PGSO}(8), \quad i = 1,2,3,
\]
which are also cyclically permuted by the triality automorphism. The existence of these three distinct maps $\lambda_i$ is a direct consequence of the triality symmetry.

If $F$ is a local field,  there are three distinct homomorphisms
\[
\lambda_i: \SO(8)(F) \twoheadrightarrow \SO(8)(F)/\mu_2(F) \hookrightarrow \mathrm{PGSO}(8)(F)
\]
with finite cokernel isomorphic to $F^{\times}/F^{\times 2}$, such that the images of any two $\lambda_i$ generate $\mathrm{PGSO}(8)(F)$. Consequently, pulling back representations along different $\lambda_i$ yields complementary information about the representation theory.

The exceptional group $G_2(F)$ appears as the stabilizer of the triality action on $\mathrm{PGSO}(8)$. This motivates our study of the triality twisted trace formula for $\mathrm{PGSO}(8)$, whose twisted elliptic endoscopic groups include $G_2$.

 The $\theta$-stable of the root data of $D_{4}$ contains three types: $G_{2}$, $\SO(4)$ and $\PGL(3)$.
According to \cite{KMRT}, the stabilizers of the triality action on $\Spin(8)$ consist of only three subgroups. We constructed the explicit twisted elliptic data of $\PGSO(8)$. 
 We also obtain the explicit elliptic endoscopic data of $\G_{2}$.
 \begin{theorem}\label{elliptic endoscopic data}
If $F$ is a local field, and
$G=\G_{2}$,  then the elliptic endoscopic data of $G$ are  \[(G_2,1,\Id),\] 
 \[(\PGL(3),s_{3},\xi_{3}),\]
 and
 \[(\SO(4),s_{4},\xi_{4}),\]
 where $s_{3}=(-1,-1,1,-1,-1,1,1)$ and $s_{4}=(1,1,1,-1,-1,-1,-1)$.
\end{theorem}
 The mapping $\xi_{3}$ and $\xi_{4}$ will appear the chapter $4$.

 Arthur \cite{A4,A5,A6} stabilized the standard trace formula, together with the stabilization of the twisted trace formula by Moeglin and Waldspurger \cite{MW2,MW3}. It provides the foundational framework for our investigation. Building upon these general results,we directly construct the twisted elliptic endoscopic data of $\PGSO(8)$ and compute the coefficient of triality twisted trace formula of $\mathrm{PGSO}(8)$.

\begin{theorem}\label{twisted trace formula of PGSO8}
Let $\widetilde{G} = \mathrm{PGSO}(8) \rtimes \theta$ with $\theta^3 = 1$, and let $f \in \widetilde{\mathcal{H}}(G, \chi)$. Then the triality twisted trace formula is given by
\[
I(f) = \widehat{S}^{G_2}(f') +\tfrac{1}{4}\widehat{S}^{\SO(4)}(f^{\prime}_{4})+ \tfrac{1}{3} \widehat{S}^{\SL_3}(f'_{3}),
\]
and its discrete part satisfies
\begin{align*}
I_{\text{disc},t}(f) &= \tr(\mathcal{I}_{\widetilde{G},t}(f))      \\
&+\tfrac{1}{6} \sum_{w \in W_{\text{reg}}(M)} |\det(w-1)_{\mathfrak{a}^{\widetilde{G}}_M}|^{-1} \tr(M_{P,t}(w,\chi)\mathcal{I}_{P,t}(\chi,f)) \\
&= \widehat{S}^{G_2}_{\mathrm{disc},t}(f') +\tfrac{1}{4}\widehat{S}^{\SO(4)}_{\disc,t}(f^{\prime}_{4})+ \tfrac{1}{3} \widehat{S}^{\SL_3}_{\mathrm{disc},t}(f'_{3}),
\end{align*}
where $M =\GL(2)$, and $f', f'_{4},f'_{3}$ denote by the Langlands-Kottwitz-Shelstad transfers of $f$.
\end{theorem}
To study the classification of automorphic representations of $\G_{2}$, we first need to understand the classification for $\PGSO(8)$. By comparing the automorphic representations of $\GSO(8)$ and $\SO(8)$, we aim to extract information about representations of $\PGSO(8)$. However, there are two main difficulties: first, the classification for $\SO(8)$ is only a weak version, namely in terms of $\mathbb{Z}_{2}$-orbits of 
A-packets. Second, the relationship between 
$\GSO(8)$ and $\SO(8)$ does not currently yield a complete classification of representations. Therefore, we consider a weak version of the classification for 
$\PGSO(8)$ by triality relations, specifically, the 
$\Out(PGSO(8))$-orbits of automorphic representations associated with stable A-parameters for $\Out(\PGSO(8))$.
We have the following coarse classification of automorphic representations of $\PGSO(8)$. 
\begin{theorem}  \label{classification of main theorem}
 There is a unique correspondence 
 \[\mathcal{L}_{\PGSO(8)}:\Psi^{\mathcal{O}}(\PGSO(8))\longrightarrow \mathcal{A}^\mathcal{O}(\PGSO(8)),\]
which maps $\psi$ to $\Pi^{\mathcal{O}}_{\psi}$, such that
\begin{equation} \label{stable character of PGSO}
    f^{\PGSO(8)}(\psi)=\sum_{\pi\in\Pi^{\mathcal{O}}_{\psi}(\varepsilon_{\psi})}\varepsilon_{\psi}(s_{\psi})\tr\pi(f)  \qquad f\in\mathcal{H}^{\mathcal{O}}(\PGSO(8))
\end{equation}
is stable. Furthermore, The correspondence $\mathcal{L}_{\PGSO(8)}$ is compatible with the two chains 
(\ref{A-parameter chain for PGSO}) and (\ref{A-automorphic chain for PGSO}), in the sense that it maps each subset in (\ref{A-parameter chain for PGSO}) onto its counterpart in (\ref{A-automorphic chain for PGSO}).

\end{theorem}
Since the out automorphism of $\G_{2}$ is trivial, then the coarse version classification of automorphic representations of $\PGSO(8)$ is enough to study the classification of $G_{2}$.
We compare the twisted trace formula of $\PGSO(8)$ with the standard trace formula of $G_{2}$, we then obtain the weak version of the classification of automorphic representations of $G_{2}$. The means is that some of the $\Out(\PGSO(8))$ stable A-parameter of $\PGSO(8)$ can parametrized the automorphic representations of $G_{2}$. 
\begin{theorem}
If $t \geq 0$ and $c \in \mathcal{C}_{\mathbb{A}}(G_2)$, then
\[
L^{2}_{\mathrm{disc},t,c}\bigl( G_2(F) \backslash G_2(\mathbb{A}) \bigr) = 0,
\]
unless
\[
(t,c) = (t(\psi), c(\psi))
\]
for some $\psi \in \Psi^{\mathcal{O}}(\PGSO(8))$. In particular, we have a decomposition
\[
L^{2}_{\mathrm{disc}}\bigl( G_2(F) \backslash G_2(\mathbb{A}) \bigr) = 
\bigoplus_{\psi \in \Psi^{\mathcal{O}}(\PGSO(8))} L^{2}_{\mathrm{disc},\psi}\bigl( G_2(F) \backslash G_2(\mathbb{A}) \bigr).
\]
\end{theorem}
 We notice that the spaces $L^{2}_{\mathrm{disc},\psi}\bigl( G_2(F) \backslash G_2(\mathbb{A}) \bigr)$ are equal to zero for some A-parameters $\psi \in \Psi^{\mathcal{O}}(\PGSO(8))$. Thus we call this classification of the automorphic representations of $G_{2}$ for weak version. Thus We need to study which A-parameters of $G_{2}$
  can parametrize the discrete spectrum of $G_{2}$ in the future.

\noindent The structure of this paper is as follows:

\begin{itemize}

\item In \S 2, We introduce the octonion algebra, Clifford algebra, the structure of $\G_{2}$ and $\Spin(8)$, and the property of triality.

    \item In \S 3, we introduce the triality structure of $\Spin(8,\mathbb{C})$ and determine the triality twisted elliptic endoscopic groups of $\mathrm{PGSO}(8)$ and the elliptic endoscopic groups of $G_2$.
    
    \item In \S 4, we construct the triality twisted elliptic endoscopic data for $\mathrm{PGSO}(8)$ and the elliptic endoscopic data for $G_2$.
    
    \item In \S 5, we develop the standard trace formula and its discrete part for $G_2$.
    
    \item In \S 6, we establish the triality twisted trace formula for $\mathrm{PGSO}(8)$ and its discrete part.
  \item In \S 7, we study the weak version classification of automorphic representation of $\PGSO(8)$, we can obtain the character relation of $\PGSO(8)$ by comparing the character relation of $\GSO(8)$ with $\SO(8)$ over local field. We then can expand the twisted trace formula for $\mathrm{PGSO}(8)$ and its discrete part.
 \item In \S 8, we establish $\psi$-component of the discrete part of triality twisted trace formula for $\mathrm{PGSO}(8)$.
    \item In \S 9, we obtain the weak version classification of automorphic representations of $\G_{2}$ by comparing the trace formula.
\end{itemize}

\section*{Acknowledgments}
The author expresses sincere gratitude to Wee Teck Gan for his encouragement and for sharing his deep insights on the exceptional group $G_2$. The author is also grateful to Chung Pang Mok for many helpful discussions. This work was supported by BJNSF(no.1242013) and the National Natural Science Foundation of China (Grant No. 12071326).

\section{Triality}

  \subsection{Octonion algebra.}

Let $F$ be a field and $\mathbb{O}$ an octonion algebra over $F$. Then $\mathbb{O}$ is an $8$-dimensional non-associative and non-commutative composition algebra with multiplicative identity, equipped with an involution $x \mapsto \bar{x}$ such that the quadratic form $N(x) = x \cdot \bar{x}$ satisfies
\[
N(x \cdot y) = N(x) \cdot N(y).
\]
Let $b_N(x, y) = N(x + y) - N(x) - N(y)$ be the associated non-degenerate symmetric bilinear form, and define the linear trace map by $\operatorname{tr}(x) = x + \bar{x}$. The octonion algebra may be realized via the Zorn model
\[
\mathbb{O}= \left\{ \begin{pmatrix} a & v \\ w^{\ast} & b \end{pmatrix} : a,b\in F,\; v\in V,\; w^{\ast}\in V^{\ast} \right\},
\]
where $V$ is a $3$-dimensional $F$-vector space and $V^{\ast}$ is its dual. The isomorphism
\[
\textstyle\bigwedge^2 V \stackrel{\sim}{\longrightarrow} V^{\ast}, \qquad
x \wedge y \longmapsto \bigl( v \mapsto x \wedge y \wedge v \bigr),
\]
where $x, y, v \in V$ and $x \wedge y \wedge v \in \bigwedge^3 V \cong F$, together with the analogous identification
\[
\textstyle\bigwedge^2 V^{\ast} \stackrel{\sim}{\longrightarrow} V,
\]
endows $\mathbb{O}$ with the multiplication
\[
\begin{pmatrix} a & v \\ w^{\ast} & b \end{pmatrix}
\cdot
\begin{pmatrix} c & x \\ y^{\ast} & d \end{pmatrix}
=
\begin{pmatrix}
ac + \langle y^{\ast}, v \rangle & a x + d v + w^{\ast} \wedge y^{\ast} \\
c w^{\ast} + b y^{\ast} + v \wedge x & bd + \langle w^{\ast}, x \rangle
\end{pmatrix}.
\]
The algebra $\mathbb{O}$ is equipped with an anti-involution
\[
\overline{\begin{pmatrix} a & v \\ w^{\ast} & b \end{pmatrix}}
= \begin{pmatrix} b & -v \\ -w^{\ast} & a \end{pmatrix},
\]
a norm form $N : \mathbb{O} \to F$ given by
\[
N\!\begin{pmatrix} a & v \\ w^{\ast} & b \end{pmatrix} = ab - \langle w^{\ast}, v \rangle,
\]
and a linear trace $\operatorname{tr} : \mathbb{O} \to F$ defined by
\[
\operatorname{tr}\!\begin{pmatrix} a & v \\ w^{\ast} & b \end{pmatrix} = a + b.
\]

Although the multiplication in $\mathbb{O}$ is not associative, the trilinear form $\operatorname{tr}(xyz)$ is well-defined and satisfies
\[
\operatorname{tr}(xyz) = \operatorname{tr}(yzx) = \operatorname{tr}(zxy),
\]
owing to the cyclic symmetry and the trace identity
\[
\operatorname{tr}(xy) = \operatorname{tr}(yx),
\]
for all $x, y, z \in \mathbb{O}$. We denote by $\mathbb{O}_0$ the $7$-dimensional subspace of trace-zero elements in $\mathbb{O}$.

The automorphism group of $\mathbb{O}$ over $F$ is isomorphic to the exceptional group $\mathrm{G}_2(F)$. It is a split simple linear algebraic group of rank $2$ that is both simply connected and adjoint. Its Lie algebra $\mathfrak{g}_2 = \operatorname{Lie}(\mathrm{G}_2)$ is identified with $\operatorname{Der}(\mathbb{O})$, the derivation algebra of $\mathbb{O}$, and is isomorphic to $\operatorname{Lie}(\operatorname{Aut}(\mathbb{O}))$. Over $\mathbb{R}$, there are two real forms of $\mathrm{G}_2$: the non-compact split form described above, and a compact form associated with a division algebra \cite{Jac}.

We record the dimensions
\[
\dim_F(\mathbb{O}) = 8, \qquad \dim_F(\mathrm{G}_2) = 14.
\]

\subsection{Root system}  

We denote by $B$ a Borel subgroup of $\mathrm{G}_2$ and fix a maximal torus $T$ contained in $B$. Relative to $(T,B)$, we obtain a system of simple roots $\{\alpha,\beta\}$ of $\mathrm{G}_2$, with $\alpha$ short and $\beta$ long. The positive root system of type $\mathrm{G}_2$ is
\[
\{\alpha,\;\beta,\;\alpha+\beta,\;2\alpha+\beta,\;3\alpha+\beta,\;3\alpha+2\beta\}.
\]
The highest root is $\beta_0 = 3\alpha + 2\beta$. The short root spaces are $3$-dimensional, while the long root spaces are $1$-dimensional (see \cite{GS}). We fix the isomorphism
\[
t \longmapsto \bigl( (2\alpha+\beta)(t),\; (\alpha+\beta)(t) \bigr)
\]
from $T$ to $\mathbb{G}_m^2$. Under this identification, any pair of characters $(\chi_1,\chi_2)$ of $F^\times$ defines a character $\chi_1 \times \chi_2$ of $T$ by composition.

Let $V_1 \subset V_2 \subset \mathbb{O}_0$ be subspaces of dimensions $1$ and $2$, respectively. Denote by $P$ and $P'$ the stabilizers of $V_2$ and $V_1$, respectively. Then $P = MN$ is a maximal parabolic subgroup of $\mathrm{G}_2$, with Levi subgroup
\[
M \cong \GL(V_2) \cong \GL(2)_s.
\]
The isomorphism $M \cong \GL(2)_s$ may be normalized so that the modulus character of $M$ is $\delta_P = |\det|^3$. Its unipotent radical $N$ is a $5$-dimensional Heisenberg group.

The subgroup $P' = M' N'$ is also a maximal parabolic subgroup of $\mathrm{G}_2$, with Levi component
\[
M' \cong \GL(V_3 / V_1) \cong \GL(2)_l,
\]
where
\[
V_3 = \{ x \in \mathbb{O}_0 : x \cdot y = 0 \text{ for all } y \in V_1 \}.
\]
The isomorphism $M' \cong \GL(2)_l$ may be normalized so that $\delta_{P'} = |\det|^5$. Its unipotent radical $N'$ is a $5$-dimensional $3$-step nilpotent group.
  
  \subsection{Clifford algebra,}

Let $V$ be a vector space over a field $F$, and let $q: V \to F$ be a quadratic form.

\begin{definition}
    A \emph{Clifford algebra} $C(V, q)$ is an $F$-algebra equipped with an $F$-linear map $J: V \to C(V, q)$ satisfying the following:
    \begin{itemize}
        \item $J(v)^2 = q(v) \cdot 1$ for all $v \in V$;
        \item (Universal Property) For any $F$-algebra $A$ and any $F$-linear map $\phi: V \to A$ such that $\phi(v)^2 = q(v) \cdot 1_A$ for all $v \in V$, there exists a unique $F$-algebra homomorphism $\psi: C(V, q) \to A$ with $\phi = \psi \circ J$.
    \end{itemize}
\end{definition}

The Clifford algebra $C(V, q)$ exists for any quadratic form $q: V \to F$ and can be constructed as the quotient of the tensor algebra
\[
T(V) = \bigoplus_{k=0}^{\infty} V^{\otimes k}
\]
by the two-sided ideal generated by all elements of the form $v \otimes v - q(v) \cdot 1$ for $v \in V$.

Let $\{e_1, \dots, e_n\}$ be an orthogonal basis of $V$ with respect to the bilinear form associated to $q$, and set $\alpha_i = q(e_i)$ for $i = 1, \dots, n$. If $\phi: V \to A$ is an $F$-linear map into an $F$-algebra $A$, then both the expression
\[
\phi(v)\phi(w) + \phi(w)\phi(v)
\]
and the associated bilinear form $b_q(v, w) = q(v+w) - q(v) - q(w)$ are symmetric and bilinear. Consequently,
\[
\phi(v)\phi(w) + \phi(w)\phi(v) = b_q(v, w) \cdot 1_A
\]
holds for all $v, w \in V$ if and only if it holds for all pairs of basis vectors $e_i, e_j$. This yields the following description.

\begin{proposition}
The Clifford algebra $C(V, q)$ is the algebra generated by $e_1, \dots, e_n$ subject to the relations:
\begin{align*}
e_i^2 &= \alpha_i \quad \text{for all } i, \\
e_i e_j + e_j e_i &= 0 \quad \text{for all } i \neq j.
\end{align*}
If $A$ is an $F$-algebra of dimension $2^n$ as a vector space, with elements $u_1, \dots, u_n$ satisfying $u_i^2 = \alpha_i$ and $u_i u_j + u_j u_i = 0$ for all $i \neq j$, then the $F$-algebra homomorphism $C(V, q) \to A$ sending $e_i \mapsto u_i$ is an isomorphism.
\end{proposition}

When $q(v) = 0$ for all $v \in V$, we have $C(V, q) \cong \bigwedge^\bullet V$, where $J: V \hookrightarrow \bigwedge^\bullet V$ is the natural embedding. The exterior algebra $\bigwedge^\bullet V$ is defined as the direct sum
\[
\bigwedge\nolimits^\bullet V = \bigoplus_{k=0}^n \bigwedge\nolimits^k V,
\]
equipped with the structure of an associative, graded-commutative $F$-algebra, with multiplication $(a, b) \mapsto a \wedge b$.

Let $\bigwedge^+ V$ and $\bigwedge^- V$ denote the even and odd parts of $\bigwedge^\bullet V$, respectively:
\[
\bigwedge\nolimits^+ V = \bigoplus_k \bigwedge\nolimits^{2k} V, \qquad
\bigwedge\nolimits^- V = \bigoplus_k \bigwedge\nolimits^{2k+1} V.
\]

Now assume $F = \mathbb{C}$ and let $q$ be the non-degenerate quadratic form $q(x) = -(x_1^2 + \cdots + x_n^2)$ on $V = \mathbb{C}^n$. Set $n = 2m$ and choose totally isotropic subspaces $W, W' \subset \mathbb{C}^n$ of dimension $m$ with $W \cap W' = 0$. Explicit bases for $W$ and $W'$ are given by
\[
w_k = \frac{i e_k + e_{k+m}}{2}, \quad 
w'_k = \frac{i e_k - e_{k+m}}{2}, \qquad k = 1, \dots, m,
\]
where $i$ denotes the imaginary unit in $\mathbb{C}$. Then
\[
q(w_k, w_j) = q(w'_k, w'_j) = 0 \quad \text{for all } k, j, \qquad
q(w_k, w'_j) = 0 \quad \text{for } k \neq j,
\]
and $q(w_k, w'_k) = \frac12$.
Thus $\mathbb{C}^n = W \oplus W'$, and there exists a unique $C(\mathbb{C}^n, q)$-module structure on $\bigwedge^\bullet W$ (see \cite[Theorem 8.2]{An}). In particular, $\bigwedge^\bullet W$ is a simple module over $C(\mathbb{C}^n, q)$.

This module structure induces a homomorphism
\[
C(\mathbb{C}^n, q) \longrightarrow \operatorname{End}_{\mathbb{C}}(\bigwedge\nolimits^\bullet W),
\]
which is in fact an isomorphism.

We now describe the structure of the even part $C^+(\mathbb{C}^n, q)$. Note that $\operatorname{End}_{\mathbb{C}}(\bigwedge^\bullet W)$ carries a natural $\mathbb{Z}/2\mathbb{Z}$-grading, where the even part consists of endomorphisms preserving parity. Hence,
\[
\operatorname{End}_{\mathbb{C}}(\bigwedge\nolimits^\bullet W)^+ \cong 
\operatorname{End}_{\mathbb{C}}({\textstyle\bigwedge^+} W) \times 
\operatorname{End}_{\mathbb{C}}({\textstyle\bigwedge^-} W),
\]
and the restricted homomorphism
\[
C^+(\mathbb{C}^n, q) \longrightarrow 
\operatorname{End}_{\mathbb{C}}({\textstyle\bigwedge^+} W) \times 
\operatorname{End}_{\mathbb{C}}({\textstyle\bigwedge^-} W)
\]
is also an isomorphism. In particular, the two $C^+(\mathbb{C}^n, q)$-modules $\bigwedge^+ W$ and $\bigwedge^- W$ are simple and non-isomorphic; these are the \emph{half-spinor modules}.

If $A \in \mathrm{O}_n(\mathbb{C})$, functoriality of the Clifford algebra implies that $A$ extends uniquely to an automorphism of $C^+(\mathbb{C}^n, q)$, still denoted $A$. For a $C^+(\mathbb{C}^n, q)$-module $M$, let $M^A$ denote the module with scalar multiplication twisted by $A$; that is, for $v \in C^+(\mathbb{C}^n, q)$ and $x \in M$, $v \cdot x := A(v)x$. Since $\bigwedge^+ W$ and $\bigwedge^- W$ are the only simple modules over $C^+(\mathbb{C}^n, q)$, two possibilities arise:
\begin{itemize}
    \item $(\bigwedge^+ W)^A \cong \bigwedge^+ W$ and $(\bigwedge^- W)^A \cong \bigwedge^- W$,
    \item $(\bigwedge^+ W)^A \cong \bigwedge^- W$ and $(\bigwedge^- W)^A \cong \bigwedge^+ W$.
\end{itemize}
The first case occurs if and only if $A \in \mathrm{SO}_n(\mathbb{C})$.

The spinor space $\bigwedge^\bullet W$ admits two important bilinear forms. Fix a nonzero element $\omega = w_1 \wedge \cdots \wedge w_m \in \bigwedge^m W$ and define a linear functional $\mathcal{C}: \bigwedge^\bullet W \to \mathbb{C}$ by
\[
x_m = \mathcal{C}(x) \cdot \omega,
\]
where $x_m$ denotes the degree-$m$ component of $x \in \bigwedge^\bullet W$. Then the canonical pairings
\[
\mathcal{N}, \overline{\mathcal{N}}: \bigwedge^\bullet W \times \bigwedge^\bullet W \to \mathbb{C}
\]
are defined by
\[
\mathcal{N}(x, y) = \mathcal{C}(x^t \wedge y), \qquad
\overline{\mathcal{N}}(x, y) = \mathcal{C}(\overline{x} \wedge y),
\]
where $\overline{x} = \iota(x)^t$ and $\iota: \bigwedge^\bullet W \to \bigwedge^\bullet W$ is the unique $\mathbb{C}$-linear map extending $w \mapsto w$ on $W$. The two pairings are related by
\[
\overline{\mathcal{N}}(x, y) = \mathcal{N}(\iota(x), y).
\]
When $m = 4$, both $\mathcal{N}$ and $\overline{\mathcal{N}}$ are non-degenerate and symmetric.
 
 \subsection{Spin groups}

In this subsection, let $F$ be either $\mathbb{R}$ or $\mathbb{C}$, and let $V$ denote $\mathbb{R}^n$ or $\mathbb{C}^n$ equipped with the quadratic form $q(x) = -\|x\|^2$.

\begin{definition}
The \emph{pin group} $\Pin_n$ is the set of elements $v \in C(V, q)$ satisfying:
\begin{itemize}
    \item $v$ is homogeneous (either even or odd in the $\mathbb{Z}/2\mathbb{Z}$-grading);
    \item $v \overline{v} = 1$;
    \item for all $u \in V$, the conjugate $v u \overline{v}$ lies in $V$.
\end{itemize}
The \emph{spin group} $\Spin_n$ is the subgroup of even elements in $\Pin_n$.
\end{definition}

Here, the conjugation $v \mapsto \overline{v}$ is the anti-automorphism of $\bigwedge^\bullet V$ defined by
\[
\overline{v} = \iota(v)^t = \iota(v^t).
\]
Thus $\Pin_n$ and $\Spin_n$ are subgroups of the group of units in $C(V, q)$.

For each $x \in \Pin_n$, define a linear map $\rho(x): V \to V$ by
\[
\rho(x)v = \iota(x) v \overline{x}, \qquad v \in V.
\]
Then $\rho(x)$ is orthogonal, since for any $v \in V$,
\begin{align*}
\|\rho(x)v\|^2 &= -q(\rho(x)v) = -(\rho(x)v)^2 \\
&= -(\iota(x) v \overline{x})^2 = -\iota(x) v \overline{x} \cdot \iota(x) v \overline{x} \\
&= -\iota(x) v (\overline{x} \iota(x)) v \overline{x} = -\iota(x) v (x \overline{x})^{-1} v \overline{x} \\
&= -\iota(x) v^2 \overline{x} = q(v) \iota(x) \overline{x} = \|v\|^2.
\end{align*}
This defines a homomorphism $\rho: \Pin_n \to \mathrm{O}_n$, which is surjective with kernel $\{\pm 1\}$. In particular, $\Spin_n$ is an index-$2$ subgroup of $\Pin_n$, and the restriction $\rho: \Spin_n \to \mathrm{SO}_n$ is also surjective with kernel $\{\pm 1\}$.

The group $\Spin_n(\mathbb{C})$ is contained in the group of units of $C^+(\mathbb{C}^n, q)$. Hence every $C^+(\mathbb{C}^n, q)$-module induces a representation of $\Spin_n(\mathbb{C})$, and likewise of $\Spin_n(\mathbb{R})$. The spinor space $\bigwedge^\bullet W$ carries a natural $C^+(\mathbb{C}^n, q)$-module structure. When $n$ is even, it decomposes as the direct sum of two half-spin representations $\bigwedge^+ W$ and $\bigwedge^- W$, both of which are irreducible. Denote these representations by
\[
\rho^+: \Spin_n \to \operatorname{End}({\textstyle\bigwedge^+} W), \qquad
\rho^-: \Spin_n \to \operatorname{End}({\textstyle\bigwedge^-} W).
\]

Since $\Spin_n$ is contained in the group of units of
\[
C^+(\mathbb{C}^n, q) \cong \operatorname{End}({\textstyle\bigwedge^+} W) \times \operatorname{End}({\textstyle\bigwedge^-} W),
\]
which is isomorphic to $\GL(\bigwedge^+ W) \times \GL(\bigwedge^- W)$, it follows that $\rho^+$ and $\rho^-$ factor through $\SL(\bigwedge^+ W)$ and $\SL(\bigwedge^- W)$, respectively. For $n \geq 3$, the commutator subgroup of $\Spin_n$ is $\Spin_n$ itself.

In particular, the two half-spin representations of $\Spin(8)$ yield an embedding
\[
\Spin_8 \hookrightarrow \SO(8, \mathbb{C}) \times \SO(8, \mathbb{C}),
\]
as the bilinear forms $\mathcal{N}$ and $\overline{\mathcal{N}}$ are invariant under the action of $\Spin_8$.

\subsection{Triality}

We take the base field $F = \mathbb{C}$ and write $\Spin(n)$ and $\SO(n)$ for $\Spin_n(\mathbb{C})$ and $\SO_n(\mathbb{C})$, respectively. The two half-spin modules $\bigwedge^+ W$ and $\bigwedge^- W$ are denoted by $S^+$ and $S^-$.

Consider three non-degenerate quadratic spaces $(V_1, q_1)$, $(V_2, q_2)$, and $(V_3, q_3)$. A bilinear map $f : V_1 \times V_2 \to V_3$ is called \emph{orthogonal} if
\[
q_3(f(v_1, v_2)) = q_1(v_1) q_2(v_2)
\]
for all $v_1 \in V_1$ and $v_2 \in V_2$. A linear map $f_v : V_1 \to V_2$ is called \emph{isometric} if $q_2(f_v(v_1)) = q_1(v_1)$ for all $v_1 \in V_1$.

A trilinear form $\mathcal{T} : V_1 \times V_2 \times V_3 \to \mathbb{C}$ induces a bilinear map $t_3 : V_1 \times V_2 \to V_3$ defined by
\[
q_3(t_3(v_1, v_2), v_3) = \mathcal{T}(v_1, v_2, v_3).
\]
This gives an isomorphism between the space of trilinear forms on $V_1 \times V_2 \times V_3$ and the space of bilinear maps from $V_1 \times V_2$ to $V_3$. Similarly, define $t_3' : V_2 \times V_1 \to V_3$ by
\[
q_3(t_3'(v_2, v_1), v_3) = \mathcal{T}(v_1, v_2, v_3).
\]
Since $q_3$ is non-degenerate, we have $t_3(v_1, v_2) = t_3'(v_2, v_1)$. In the same way, define bilinear maps
\begin{align*}
t_1, t_1' &: V_2 \times V_3 \to V_1, \\
t_2, t_2' &: V_1 \times V_3 \to V_2,
\end{align*}
by the relations
\[
q_1(t_1(v_2, v_3), v_1) = \mathcal{T}(v_1, v_2, v_3) = q_1(t_1'(v_3, v_2), v_1),
\]
\[
q_2(t_2(v_1, v_3), v_2) = \mathcal{T}(v_1, v_2, v_3) = q_2(t_2'(v_3, v_1), v_2).
\]
Then $t_1(v_2, v_3) = t_1'(v_3, v_2)$ and $t_2(v_1, v_3) = t_2'(v_3, v_1)$.

\begin{definition}
A \emph{triality} is a quadruple $(V_1, V_2, V_3, \mathcal{T})$, where $(V_1, q_1)$, $(V_2, q_2)$, $(V_3, q_3)$ are non-degenerate quadratic spaces of positive dimension, and $\mathcal{T} : V_1 \times V_2 \times V_3 \to \mathbb{C}$ is a trilinear form such that each of the induced bilinear maps $t_1, t_2, t_3$ is orthogonal.
\end{definition}

\begin{remark}
The orthogonality of each $t_i$ ($i = 1, 2, 3$) forces $\dim V_1 = \dim V_2 = \dim V_3$. In fact, it suffices that only one of the $t_i$ be orthogonal. Trialities are exceedingly rare; the \emph{dimension of a triality} is the common dimension of the $V_i$.
\end{remark}

\begin{theorem}[Hurwitz]
A triality can exist only in dimensions $1$, $2$, $4$, and $8$.
\end{theorem}
\begin{proof}
See \cite[Theorem 10.8]{An}.
\end{proof}

Consider $V_1 = \mathbb{C}^8$, $V_2 = S^+$, $V_3 = S^-$, where the quadratic forms are given by $q(x) = -(x_1^2 + \cdots + x_8^2)$ for $x = (x_1, \dots, x_8) \in \mathbb{C}^8$, and $\mathcal{N}$ restricted to $S^+$ and $S^-$, respectively. Define the bilinear map $t_3 : \mathbb{C}^8 \times S^+ \to S^-$ via the Clifford module structure
\[
t_3(v, x) = v \cdot x.
\]
Note that $v \cdot x \neq x \cdot v$ in general, but we may define an orthogonal bilinear map $t_3'(x, v) = v \cdot x = t_3(v, x)$. By \cite[Lemma 9.21]{An}, we have
\[
\mathcal{N}(v \cdot x) = q(v) \mathcal{N}(x) \quad \text{for all } v \in \mathbb{C}^8,\ x \in S^+,
\]
so $t_3$ is orthogonal. Define the trilinear form $\mathcal{T} : \mathbb{C}^8 \times S^+ \times S^- \to \mathbb{C}$ by
\[
\mathcal{T}(v, x, y) = \mathcal{N}(t_3(v, x), y).
\]
This induces two additional bilinear maps
\begin{align*}
t_1 &: S^+ \times S^- \to \mathbb{C}^8, \\
t_2 &: \mathbb{C}^8 \times S^- \to S^+,
\end{align*}
satisfying
\[
q(v, t_1(x, y)) = \mathcal{T}(v, x, y) = \mathcal{N}(t_2(v, y), x),
\]
for all $v \in \mathbb{C}^8$, $x \in S^+$, $y \in S^-$. We also set $t_1(x, y) = t_1'(y, x)$ and $t_2(v, y) = t_2'(y, v)$.

We claim that $(\mathbb{C}^8, S^+, S^-, \mathcal{T})$ is a triality. From \cite[Lemma 9.21]{An}, we have $\mathcal{N}(v \cdot x, y) = \mathcal{N}(x, v \cdot y)$, so by non-degeneracy of $\mathcal{N}$,
\[
t_2(v, y) = v \cdot y,
\]
which is orthogonal. To show $t_1$ is orthogonal, suppose $\mathcal{N}(x) \neq 0$ and define
\[
f_x : S^- \to \mathbb{C}^8,\quad f_x(y) = t_1(x, y), \qquad
g_x : \mathbb{C}^8 \to S^-,\quad g_x(v) = v \cdot x.
\]
Then $\mathcal{N}(g_x(v)) = q(v) \mathcal{N}(x)$ for all $v \in \mathbb{C}^8$. For any $v, v_1 \in \mathbb{C}^8$,
\begin{align*}
q(f_x g_x(v), v_1) &= \mathcal{N}(g_x(v), g_x(v_1)) \\
&= \frac{1}{2} \bigl( \mathcal{N}((v + v_1) \cdot x) - \mathcal{N}(v \cdot x) - \mathcal{N}(v_1 \cdot x) \bigr) \\
&= q(v, v_1) \mathcal{N}(x).
\end{align*}
Hence $f_x \circ g_x = \mathcal{N}(x) \cdot \operatorname{Id}_{\mathbb{C}^8}$. Since $\dim \mathbb{C}^8 = \dim S^-$ and $\mathcal{N}(x) \neq 0$, it follows that $g_x \circ f_x = \mathcal{N}(x) \cdot \operatorname{Id}_{S^-}$. Therefore,
\[
q(t_1(x, y)) = q(f_x(y), f_x(y)) = \mathcal{N}(g_x f_x(y), y) = \mathcal{N}(x) \mathcal{N}(y),
\]
so $t_1$ is orthogonal. Thus $(\mathbb{C}^8, S^+, S^-, \mathcal{T})$ is a triality.

For simplicity, we write $x y = t_1(x, y) = t_1'(y, x)$ for $x \in S^+$, $y \in S^-$.

\begin{lemma}\label{compute tool}
Let $(V_1, V_2, V_3, \mathcal{T})$ be a triality. For any permutation $\{i, j, k\} = \{1, 2, 3\}$, and for all $v_i, v_i' \in V_i$, $v_k \in V_k$, we have:
\begin{align}
v_i (v_i' v_k) + v_i' (v_i v_k) &= 2 q_i(v_i, v_i') v_k, \tag{1}\\
v_i (v_i v_k) &= q_i(v_i) v_k. \tag{2}
\end{align}
\end{lemma}

\begin{proof}
The two identities are equivalent via polarization. It suffices to prove (2). Consider the maps
\[
v_i \cdot (-) : V_j \to V_k, \quad v_i \cdot (-) : V_k \to V_j,
\]
defined by $v_i \cdot v_j = v_i v_j$ and $v_i \cdot v_k = v_i v_k$. Since $q_j(v_i v_j) = q_i(v_i) q_j(v_j)$ for all $v_j \in V_j$, we have for any $v_k, v_k' \in V_k$:
\begin{align*}
q_k(v_i (v_i v_k), v_k') &= \mathcal{T}(v_i, v_i v_k, v_k') \\
&= q_j(v_i v_k, v_i v_k') \\
&= q_i(v_i) q_k(v_k, v_k').
\end{align*}
By non-degeneracy of $q_k$, it follows that $v_i (v_i v_k) = q_i(v_i) v_k$.
\end{proof}

\begin{definition}
Let $\tau = (V_1, V_2, V_3, \mathcal{T})$ be a triality. A \emph{triality automorphism} of $\tau$ is a triple $(A_1, A_2, A_3)$, where each $A_i$ is an orthogonal automorphism of $V_i$, such that
\[
\mathcal{T}(A_1(v_1), A_2(v_2), A_3(v_3)) = \mathcal{T}(v_1, v_2, v_3)
\]
for all $v_1 \in V_1$, $v_2 \in V_2$, $v_3 \in V_3$.
\end{definition}

The set of all such triality automorphisms, with group structure given by component-wise composition, is called the \emph{triality automorphism group} of $\tau$, denoted $\Aut^0(\tau)$. Thus $\Aut^0(\tau)$ is a subgroup of $\mathrm{O}(V_1, q_1) \times \mathrm{O}(V_2, q_2) \times \mathrm{O}(V_3, q_3)$.

\begin{proposition} \label{equivalent def triality}
Let $\tau = (V_1, V_2, V_3, \mathcal{T})$ be a triality, and $(A_1, A_2, A_3)$ a triple of orthogonal automorphisms of $V_1, V_2, V_3$ respectively. Then $(A_1, A_2, A_3)$ is a triality automorphism if and only if
\[
t_3(A_1(v_1), A_2(v_2)) = A_3(t_3(v_1, v_2))
\]
for all $v_1 \in V_1$, $v_2 \in V_2$.
\end{proposition}

\begin{proof}
Since $\tau$ is a triality, we have
\begin{align*}
\mathcal{T}(v_1, v_2, v_3) &= q_3(t_3(v_1, v_2), v_3) = q_3(A_3(t_3(v_1, v_2)), A_3(v_3)), \\
\mathcal{T}(A_1(v_1), A_2(v_2), A_3(v_3)) &= q_3(t_3(A_1(v_1), A_2(v_2)), A_3(v_3)).
\end{align*}
By non-degeneracy of $q_3$, the equality
\[
q_3(A_3(t_3(v_1, v_2)), A_3(v_3)) = q_3(t_3(A_1(v_1), A_2(v_2)), A_3(v_3))
\]
holds for all $v_3 \in V_3$ if and only if
\[
t_3(A_1(v_1), A_2(v_2)) = A_3(t_3(v_1, v_2)).
\]
\end{proof}

The center of $\Spin(8)$ is $C(\Spin(8)) = \{\pm 1, \pm \eta\}$, where $\eta = e_1 \cdots e_8 \in \Spin(8)$. Since $e_i$ anti-commutes with $e_j$ for $i \neq j$ but commutes with itself, we have $\eta e_i = (-1)^7 e_i \eta = -e_i \eta$; hence $\eta v = -v \eta$ for all $v \in \mathbb{C}^8$. As every element of $\Spin(8)$ is a product of an even number of unit vectors in $\Pin(8)$, it follows that $\eta$ lies in the center of $\Spin(8)$.

The complex group $\Spin(8)$ admits three non-isomorphic $8$-dimensional irreducible orthogonal representations:
\begin{itemize}
    \item the standard representation $\rho: \Spin(8) \to \SO(8)$;
    \item the two half-spin representations $\rho^+: \Spin(8) \to \GL(S^+)$ and $\rho^-: \Spin(8) \to \GL(S^-)$, which factor through $\SO(S^+, \mathcal{N})$ and $\SO(S^-, \mathcal{N})$, respectively.
\end{itemize}
These satisfy
\begin{align*}
\rho|_{C(\Spin(8))} &= 1; \\
\rho^+(1) = \rho^+(\eta) &= 1, \quad \rho^+(-1) = \rho^+(-\eta) = -1; \\
\rho^-(1) = \rho^-(-\eta) &= 1, \quad \rho^-(-1) = \rho^-(\eta) = -1.
\end{align*}

\begin{proposition} \label{isomorphim}
The map
\[
(\rho, \rho^+, \rho^-): \Spin(8) \hookrightarrow \SO(8) \times \SO(S^+, \mathcal{N}) \times \SO(S^-, \mathcal{N})
\]
is an embedding whose image is the triality automorphism group $\Aut^0(\mathbb{C}^8, S^+, S^-, \mathcal{T})$.
\end{proposition}

\begin{proof}
For $a \in \Spin(8)$, define
\begin{align*}
\rho(a)v &= a v \overline{a}, \quad v \in \mathbb{C}^8, \\
\rho^+(a)x &= a x, \quad x \in S^+, \\
\rho^-(a)y &= a y, \quad y \in S^-,
\end{align*}
where the products are taken in the Clifford algebra. Then
\[
t_3(\rho(a)v, \rho^+(a)x) = (a v \overline{a})(a x) = a(v x) = \rho^-(a)(t_3(v, x)).
\]
Hence $(\rho(a), \rho^+(a), \rho^-(a))$ is a triality automorphism.

To show surjectivity, note that the standard representation $\rho: \Spin(8) \to \SO(8)$ is surjective with kernel $\{\pm 1\}$. Let $\sigma: \Aut^0(\mathbb{C}^8, S^+, S^-, \mathcal{T}) \to \SO(8)$ be the projection onto the first factor. We show $\ker \sigma$ has order $2$.

Let $(A_1, A_2, A_3) \in \Aut^0(\tau)$. By functoriality, $A_1$ extends uniquely to an automorphism of $C^+(\mathbb{C}^8, q)$. The identities
\[
A_1(v) A_2(x) = A_3(v x), \qquad A_1(v) A_3(y) = A_2(v y)
\]
imply that $A_2 \oplus A_3: S^+ \oplus S^- \to S^+ \oplus S^-$ is an isomorphism of $C(\mathbb{C}^8, q)$-modules. Hence $A_1 \in \SO(8)$.

Now consider $(\Id_{\mathbb{C}^8}, A_2, A_3) \in \ker \sigma$. Then $A_3(v x) = v A_2(x)$ for all $v \in \mathbb{C}^8$, $x \in S^+$, so $A_2 \oplus A_3$ is a $C(\mathbb{C}^8, q)$-module endomorphism of the simple module $\bigwedge^\bullet W = S^+ \oplus S^-$. By Schur's lemma, $A_2 \oplus A_3$ is a scalar map. Since $A_2, A_3$ are orthogonal, we must have $A_2 = A_3 = \pm 1$. Therefore $\ker \sigma \cong \{\pm 1\}$, and the embedding is an isomorphism.
\end{proof}
 
 \section{Triality and Arthur Parameters}

We fix an isomorphism of quadratic spaces
\[
(\mathbb{C}^8, q) \cong (S^+, \mathcal{N}) \cong (S^-, \mathcal{N}),
\]
which yields an embedding $\Spin(8) \hookrightarrow \SO(8) \times \SO(8) \times \SO(8)$. Let $\theta_3$ denote the cyclic permutation of order $3$ acting on $\SO(8)^3$ by rotating the factors. A natural question is whether this action lifts to an automorphism of $\Spin(8)$. The following lemma answers this affirmatively.

\begin{lemma} \label{triality action}
There exists an automorphism $\theta$ of order $3$ on $\Spin(8)$ such that the following diagram commutes:
\[
\begin{tikzcd}
\Spin(8) \ar[r, "\theta"] \ar[d, "i"] & 
\Spin(8) \ar[d, "i"] \\
\SO(8)^3 \ar[r, "\theta_3"] & \SO(8)^3 
\end{tikzcd}
\]
where $i = (\rho, \rho^+, \rho^-)$.
\end{lemma}
\begin{proof}
Choose unit vectors $v_1 \in \mathbb{C}^8$, $x_1 \in S^+$ with $q(v_1) = \mathcal{N}(x_1) = 1$, and set $y_1 = v_1 x_1 \in S^-$. Then $\mathcal{N}(y_1) = 1$, and by Lemma \ref{compute tool},
\[
v_1 y_1 = v_1(v_1 x_1) = q(v_1) x_1 = x_1.
\]
Define the $\mathbb{C}$-linear isometries
\begin{align*}
\alpha_1 = v_1 \cdot (-) &: S^+ \to S^-, \quad & \beta_1 = v_1 \cdot (-) &: S^- \to S^+, \\
\alpha_2 = x_1 \cdot (-) &: \mathbb{C}^8 \to S^-, \quad & \beta_2 = x_1 \cdot (-) &: S^- \to \mathbb{C}^8, \\
\alpha_3 = y_1 \cdot (-) &: \mathbb{C}^8 \to S^+, \quad & \beta_3 = y_1 \cdot (-) &: S^+ \to \mathbb{C}^8.
\end{align*}
Let $R_{v_1}$ denote the reflection in $\mathbb{C}^8$ across the hyperplane orthogonal to $v_1$:
\[
R_{v_1}(v) = v - 2q(v_1, v)v_1.
\]
Using Lemma \ref{compute tool}, one verifies
\begin{align*}
v_1(x_1 v) &= y_1(-R_{v_1} v), \\
v_1(y_1 v) &= x_1(-R_{v_1} v), \\
(v_1 x)(v_1 y) &= -R_{v_1}(x y).
\end{align*}
Define the involution $\iota_1 \in \Aut^0(\tau)$ by
\[
\iota_1(v, x, y) = (-R_{v_1} v,\; v_1 y,\; v_1 x).
\]
Similarly, define $\iota_2$ corresponding to the transposition $(13)$:
\[
\iota_2(v, x, y) = (x_1 y,\; -R_{x_1} x,\; x_1 v).
\]
Then $\iota_1^2 = \iota_2^2 = 1$, and $\theta' = \iota_2 \iota_1$ satisfies
\[
\theta'(v, x, y) = (x_1(v_1 x),\; y_1(x_1 y),\; v_1(y_1 v)),
\]
with $\theta'^3 = 1$. The map $\sigma_1(a) = v_1 a$ on $\Spin(8)$ makes the following diagram commute:
\[
\begin{tikzcd}
\Spin(8) \ar[r, "\sigma_1"] \ar[d, "i"] & \Spin(8) \ar[d, "i"] \\
\Aut^0(\tau) \ar[r, "\iota_1"] & \Aut^0(\tau)
\end{tikzcd}
\]
 Thus we define $\theta(a) = i^{-1}(\theta'(i(a)))$.

To match the $\SO(8)^3$ action, take $h = (\Id_{\SO(8)},\, x_1 v_1,\, y_1 v_1)$. Then
\[
\begin{tikzcd}
\SO(8) \times \SO(S^+) \times \SO(S^-) \ar[r, "\theta'"] \ar[d, "h"] & 
\SO(8) \times \SO(S^+) \times \SO(S^-) \ar[d, "h"] \\
\SO(8)^3 \ar[r, "\theta_3"] & \SO(8)^3
\end{tikzcd}
\]
commutes, and the result follows.
\end{proof}

\begin{remark}
The elements $\iota_1$ and $\theta'$ generate a copy of $S_3$ acting on $\Aut^0(\tau)$, and hence on $\Spin(8)$ via $i$. This extends the $\mathbb{Z}/3\mathbb{Z}$ action to a full $S_3$-action.
\end{remark}

\begin{remark}
The embedding $f = (f_1, f_2, f_3): \Spin(8) \hookrightarrow \SO(8)^3$ satisfies
\[
f(\theta(a)) = (f_2(a), f_3(a), f_1(a)).
\]
In particular, $f_2(a) = f_1(\theta(a))$ and $f_3(a) = f_1(\theta^2(a))$, so the triple is determined by $f_1$ and $\theta$.
\end{remark}

\begin{remark}
The fixed points of $\theta_3$ on $\SO(8)^3$ are exactly the diagonal copies of $\SO(8)$. In particular, $\PGSO(8)$ (whose dual group is $\Spin(8)$) is not an endoscopic group of $\SO(8)^3$.
\end{remark}

Now let $F$ be a local field and $G$ a connected reductive group. An \emph{Arthur parameter} is an $L$-homomorphism
\[
\psi: L_F \times \SL(2) \to {}^L G.
\]
Here the local Langlands group is
\[
L_{F} = 
\begin{cases} 
W_{F}, & \text{if } F \text{ is archimedean,} \\[4pt]
W_{F}\times \SU(2), & \text{if } F \text{ is nonarchimedean,}
\end{cases}
\]
where $W_F$ denotes the Weil group of the local field $F$.

Let $\Psi(G)$ denote the set of $\widehat{G}$-conjugacy classes of such $\psi$. By Lemma \ref{triality action} and the preceding remarks, we obtain the following.

\begin{theorem}
There is a canonical map
\[
f_*: \Psi(\PGSO(8)) \to \Psi(\SO(8))^3
\]
defined by
\[
f_*(\psi) = (f_1 \circ \psi,\; f_1 \circ (\theta \circ \psi),\; f_1 \circ (\theta^2 \circ \psi)).
\]
\end{theorem}

\begin{definition}
A triple $(\psi_1, \psi_2, \psi_3) \in \Psi(\SO(8))^3$ is called a \emph{triality Arthur parameter} if it lies in the image of $f_*$. Denote by $\Psi^\tau(\PGSO(8))$ the set of all triality Arthur parameters.
\end{definition}

By construction, $\Psi^\tau(\PGSO(8)) = \Image(f_*)$.

\subsection{Triality for Octonion Algebras}

Let $V_1 = V_2 = V_3 = \mathbb{O}$ be $8$-dimensional vector spaces, and define the trilinear form $\mathcal{T}(x,y,z) = \operatorname{tr}(xyz)$. Then $(\mathbb{O}, \mathbb{O}, \mathbb{O}, \operatorname{tr})$ forms a triality.

Since $\operatorname{tr}(xy) = \operatorname{tr}(yx)$ for all $x, y \in \mathbb{O}$, we have the cyclic symmetry
\[
\operatorname{tr}(xyz) = \operatorname{tr}(yzx) = \operatorname{tr}(zxy), \qquad \forall x, y, z \in \mathbb{O}.
\]
Let $b_N(x, y) = N(x + y) - N(x) - N(y)$ be the polar form of $N$, and note that $\overline{xy} = \bar{y} \bar{x}$. Then
\[
\operatorname{tr}(xy) = b_N(x, t(y)), \quad \text{where } t(y) = \bar{y}.
\]
Define $t_1(y, z) = \overline{yz}$. Then
\[
\operatorname{tr}(xyz) = b_N(x, t_1(y, z)).
\]
Define quadratic forms
\[
q_1(y, z) = b_N(\bar{y}, \bar{z}), \qquad q_2(x, y) = b_N(x, y), \qquad q_3(x, y) = \tfrac{1}{2} b_N(x, y).
\]
Using the composition property $N(xy) = N(x)N(y)$, one verifies
\[
q_1(t_1(y, z)) = q_2(y) q_3(z).
\]
Similarly, define
\[
t_2(x, y) = \bar{x} y, \qquad t_3(x, y) = 2 \bar{x} y,
\]
and check
\[
q_2(t_2(x, z)) = q_1(x) q_3(z), \qquad q_3(t_3(x, y)) = q_1(x) q_2(y).
\]
Thus $t_1, t_2, t_3$ are orthogonal, and $(\mathbb{O}, \mathbb{O}, \mathbb{O}, \operatorname{tr})$ is a triality.

Since $\Spin(8) \cong \Aut^0(\tau)$, we have
\[
\Spin(8) \cong \{(A, B, C) \in \SO(\mathbb{O})^3 : \operatorname{tr}(A(x) B(y) C(z)) = \operatorname{tr}(xyz), \ \forall x, y, z \in \mathbb{O} \}.
\]
An $S_3$-action on $\Spin(8)$ is defined as follows. Let $x_1 = 1 \in \mathbb{O}$. For $A \in \SO(\mathbb{O})$, define $\widehat{A}(x) = \overline{A(\bar{x})}$. Then the involutions
\[
\sigma_1(A, B, C) = (\widehat{A}, \widehat{C}, \widehat{B}), \qquad \sigma_2(A, B, C) = (\widehat{C}, \widehat{B}, \widehat{A}),
\]
lift the transpositions in $S_3$, and $\theta = \sigma_2 \sigma_1$ satisfies $\theta(A, B, C) = (B, C, A)$ with $\theta^3 = 1$.

For $g \in \Aut(\mathbb{O}) = \mathrm{G}_2$, there exists $x_1 \in \mathbb{O}$ such that $g(x) = x_1^{-1} x x_1$ (since $\Out(\mathrm{G}_2) = 1$). Then
\[
\operatorname{tr}(g(x) g(y) g(z)) = \operatorname{tr}(xyz), \qquad \forall x, y, z \in \mathbb{O},
\]
so $\mathrm{G}_2 \subset \Spin(8)$. The action of $\mathrm{G}_2$ fixes $1 \in \mathbb{O}$, and its orthogonal complement is the $7$-dimensional standard representation. Moreover,
\[
\Spin(8)^\theta = \mathrm{G}_2.
\]

Define a new multiplication on $\mathbb{O}$ by $x \star y = \bar{x} \cdot \bar{y}$. Then $(\mathbb{O}, \star)$ is an $8$-dimensional symmetric composition algebra, satisfying
\[
(x \star y) \star x = x \star (y \star x) = N(x) y, \quad N(x \star y) = N(x) N(y), \quad b_N(x \star y, z) = b_N(x, y \star z).
\]
By Proposition \ref{equivalent def triality}, we obtain
\[
\Spin(8) \cong \{(A, B, C) \in \SO(\mathbb{O}, N)^3 : A(x) \star B(y) = C(x \star y), \ \forall x, y \in \mathbb{O} \}.
\]

By the classification of symmetric composition algebras \cite[(35.10) Corollary]{KMRT}, there exists another split symmetric composition algebra $S$ over $\mathbb{C}$ of dimension $8$. This is realized on the space $\mathbb{M}^0$ of $3 \times 3$ trace-zero matrices. Let $\mathbb{M} = M_3(\mathbb{C})$, and for $x \in \mathbb{M}$, let
\[
P_{\mathbb{M}, x}(X) = X^3 - \operatorname{tr}_{\mathbb{M}}(x) X^2 + S_{\mathbb{M}}(x) X - N_{\mathbb{M}}(x) 1_{\mathbb{M}}
\]
be the generic characteristic polynomial. Then $S_{\mathbb{M}}$ is a quadratic form and $N_{\mathbb{M}}$ a cubic form. Let
\[
b_{S_{\mathbb{M}}}(x, y) = S_{\mathbb{M}}(x + y) - S_{\mathbb{M}}(x) - S_{\mathbb{M}}(y),
\]
so that $b_{S_{\mathbb{M}}}(x, x) = 2 S_{\mathbb{M}}(x)$. By \cite[Lemma 34.14]{KMRT},
\[
\operatorname{tr}_{\mathbb{M}}(xy) = \operatorname{tr}_{\mathbb{M}}(x) \operatorname{tr}_{\mathbb{M}}(y) - b_{S_{\mathbb{M}}}(x, y).
\]
Since $\mathbb{M} = \mathbb{C} 1 \oplus \mathbb{M}^0$ and $\operatorname{tr}_{\mathbb{M}}(x) = 0$ for $x \in \mathbb{M}^0$, we have
\[
\operatorname{tr}_{\mathbb{M}}(xy) = -b_{S_{\mathbb{M}}}(x, y), \qquad \forall x, y \in \mathbb{M}^0,
\]
so $b_{S_{\mathbb{M}}}$ restricts to a non-degenerate quadratic form on $\mathbb{M}^0$.

Let $\omega$ be a primitive cube root of unity and set $\mu = \dfrac{1 - \omega}{3}$. Define a multiplication on $\mathbb{M}^0$ by
\[
x \star y = \mu xy + (1 - \mu) yx - \operatorname{tr}_{\mathbb{M}}(yx) 1_{\mathbb{M}}.
\]
This is the Okubo algebra structure \cite{Ok}.

\begin{proposition}[\cite{KMRT}, Proposition 34.19] \label{symmetric composition M}
The algebra $(\mathbb{M}^0, \star)$ is a symmetric composition algebra with norm $n(x) = -\dfrac{1}{3} S_{\mathbb{M}}(x)$.
\end{proposition}

\section{Elliptic Endoscopic Data}
From the arguments in Proposition \ref{symmetric composition M}, we observe that $(\mathbb{M}^0, \mathbb{M}^0, \mathbb{M}^0, \operatorname{tr}_{\mathbb{M}}((x \star y) z))$ forms a triality. Since $\mathbb{M}^0$ is an $8$-dimensional symmetric composition algebra, take $V_1 = V_2 = V_3 = \mathbb{M}^0$ and $\mathcal{T}(x, y, z) = \operatorname{tr}_{\mathbb{M}}((x \star y) z)$. Using the identities
\[
\operatorname{tr}_{\mathbb{M}}(xy) = \operatorname{tr}_{\mathbb{M}}(yx), \qquad 
b_n(x, y) = \tfrac{1}{3} \operatorname{tr}_{\mathbb{M}}(xy), \qquad 
b_n(x \star y, z) = b_n(x, y \star z),
\]
we obtain
\[
\operatorname{tr}_{\mathbb{M}}((x \star y) z) = \operatorname{tr}_{\mathbb{M}}(x (y \star z)) = \operatorname{tr}_{\mathbb{M}}(y (z \star x)).
\]
Define quadratic forms
\[
q_1(x, y) = \tfrac{3}{2} b_n(x, y), \qquad q_2(x, y) = 3 b_n(x, y), \qquad q_3(x, y) = b_n(x, y).
\]
Then the bilinear maps
\[
t_1(x, y) = 2x \star y, \qquad t_2(x, y) = x \star y, \qquad t_3(x, y) = 3x \star y,
\]
satisfy
\begin{align*}
\mathcal{T}(x, y, z) &= q_1(x, t_1(y, z)) = q_2(y, t_2(z, x)) = q_3(z, t_3(x, y)), \\
q_1(t_1(y, z)) &= q_2(y) q_3(z), \\
q_2(t_2(x, z)) &= q_1(x) q_3(z), \\
q_3(t_3(x, y)) &= q_1(x) q_2(y),
\end{align*}
for all $x, y, z \in \mathbb{M}^0$. Thus $(\mathbb{M}^0, \mathbb{M}^0, \mathbb{M}^0, \operatorname{tr}_{\mathbb{M}}((x \star y) z))$ is a triality.

The associated spin group is
\[
\operatorname{Spin}(\mathbb{M}^0, \star, n) = \{(A, B, C) \in \SO(\mathbb{M}^0, n)^3 : A(x \star y) = B(x) \star C(y), \ \forall x, y \in \mathbb{M}^0 \}.
\]
By \cite[\S 35B]{KMRT}, this is a simply-connected algebraic group of type $D_4$, hence isomorphic to $\Spin(8)$. The automorphism group of $(\mathbb{M}^0, \star)$ is
\[
\Ad(\PGL(3)) \subset \SO(\mathbb{M}^0, n),
\]
acting by conjugation. Thus $\PGL(3)$ acts on $\mathbb{M}^0$ via the adjoint representation, inducing an action on $\operatorname{Spin}(\mathbb{M}^0, \star, n)$.

Define a third-order automorphism
\[
\theta: \operatorname{Spin}(\mathbb{M}^0, \star, n) \to \operatorname{Spin}(\mathbb{M}^0, \star, n),
\]
permuting the coordinates in $\SO(\mathbb{M}^0, n)^3$. Then
\[
\operatorname{Spin}(\mathbb{M}^0, \star, n)^\theta = \PGL(3),
\]
and the three projections $\iota_i: \operatorname{Spin}(\mathbb{M}^0, \star, n) \to \SO(\mathbb{M}^0, n)$ satisfy $\ker(\iota_i) = \mu_2$, with $\theta$ cyclically permuting the $\iota_i$.

Recall the short exact sequence
\[
1 \to \PGSO(8) \to \Aut(\Spin(8)) \to S_3 \to 1.
\]
A symmetric composition algebra $(V, \star, Q)$ yields a splitting of this sequence over the order-$3$ subgroup $A_3 \subset S_3$. The two symmetric composition algebras $(\mathbb{O}, \star, N)$ and $(\mathbb{M}^0, \star, n)$ give two such splittings, with fixed subgroups $\mathrm{G}_2$ and $\PGL(3)$, respectively. By \cite[Corollary 35.10]{KMRT}, these are the only two splittings up to inner conjugation over $\mathbb{C}$.

\subsection{Triality twisted elliptic endoscopic data}

Let $\widetilde{G} = \PGSO(8) \rtimes \theta^\vee$, where $\theta^\vee$ is the triality automorphism dual to $\theta$ on $\Spin(8, \mathbb{C})$. For simplicity, we denote $\theta^\vee$ by $\theta$.

The groups $G_2$ and $\SL(3)$ are $\theta$-twisted endoscopic groups for $\PGSO(8)$. There exist $L$-embeddings:
\begin{align*}
\xi: {}^L G' = G_2(\mathbb{C}) \times \Gamma_F &\hookrightarrow {}^L G = \Spin(8, \mathbb{C}) \times \Gamma_F, \\
\Ad: \PGL(3, \mathbb{C}) \times \Gamma_F &\hookrightarrow \Spin(8, \mathbb{C}) \times \Gamma_F,
\end{align*}
obtained by inflating the embeddings
\[
\xi: G_2(\mathbb{C}) \hookrightarrow \Spin(8, \mathbb{C}), \qquad 
\Ad: \PGL(3, \mathbb{C}) \to \Spin(8, \mathbb{C})
\]
to the full $L$-groups.

Consider the semi-direct product
\[
\widetilde{G}^+ = \PGSO(8) \times \langle \theta \rangle,
\]
where $\langle \theta \rangle$ is the cyclic group of order $3$. Let $w$ be a primitive third root of unity, and set
\[
s_0 = \exp\!\Big( \frac{2\pi}{3} \cdot \frac{e_2 e_6 - e_3 e_7}{2} \Big) \in \Spin(8, \mathbb{C}), \qquad 
s = s_0 \rtimes \theta \in \widehat{\widetilde{G}}.
\]
Then the image of $s_0$ under the standard representation is
\[
\rho(s_0) = \diag(1, w, w^{-1}, 1, 1, w^{-1}, w, 1) \in \SO(8, \mathbb{C}),
\]
where $\{e_1, e_2, \dots, e_8\}$ is the standard orthonormal basis of $\mathbb{C}^8$, and $e_2 e_6$ denotes Clifford multiplication.

The connected centralizer of $s$ in $\widehat{\PGSO}(8) \rtimes \theta = \Spin(8, \mathbb{C}) \rtimes \theta$ is $\Ad(\PGL(3, \mathbb{C}))$. The triple $(\SL(3), s, \Ad)$ is an elliptic endoscopic datum for $\widetilde{G}$, in the sense of \cite[p.16]{KS}.

Similarly, take $s_1 = 1 \rtimes \theta \in \Spin(8, \mathbb{C}) \rtimes \theta$, where $1$ is the identity element in $\Spin(8, \mathbb{C})$. The connected centralizer of $s_1$ in $\Spin(8, \mathbb{C})$ is
\[
\widehat{G}_2 = G_2(\mathbb{C}) = Z_{\Spin(8, \mathbb{C})}(s_1)^\circ,
\]
so $(G_2, s_1, \xi)$ is also an endoscopic datum for $\widetilde{G}$. We call these the \emph{simple endoscopic data} for $\widetilde{G}$.

An endoscopic datum for $\widetilde{G}$ is a triple $(G', s, \xi)$, where $G'$ is a quasisplit group over $F$, $s$ is a semisimple element in $\widehat{\widetilde{G}}$ with connected centralizer $\widehat{G}'$, and $\xi$ is an $L$-embedding ${}^L G' \hookrightarrow {}^L \widetilde{G}$. The datum is \emph{elliptic} if the $\Gamma$-invariant part of $Z(\widehat{G}')$ is finite.

Isomorphisms between endoscopic data are defined as in \cite[p.18]{KS}. Let $\Aut_{\widetilde{G}}(G')$ be the group of self-isomorphisms of the datum $G'$, and define
\[
\operatorname{Out}_{\widetilde{G}}(G') = \Aut_{\widetilde{G}}(G') / \operatorname{Int}_{\widetilde{G}}(G'),
\]
which we identify with the group of $F$-automorphisms of $G'$ preserving a splitting $(B', T')$. For the simple data $G_2$ and $\SL(3)$, we have $\operatorname{Out}_{\widetilde{G}}(G') = 1$.

Let $\mathcal{E}(\widetilde{G})$ denote the set of isomorphism classes of endoscopic data for $\widetilde{G}$, and $\mathcal{E}_{\text{ell}}(\widetilde{G})$ the subset of elliptic classes. Let $\mathcal{E}_{\text{simp}}(\widetilde{G})$ be the set of simple endoscopic data, which consists precisely of $G_2$ and $\SL(3)$.

The triple $(\SO(4), s_4' \rtimes \theta, \xi_4')$ is a $\theta$-twisted elliptic endoscopic datum for $\PGSO(8)$, but it is not a simple endoscopic datum. Here
\[
s_4' = \exp\!\Big( -\frac{\pi}{4} e_3 e_5 \Big) \exp\!\Big( -\frac{\pi}{4} e_4 e_6 \Big) \exp\!\Big( -\frac{\pi}{4} e_1 e_3 \Big) \exp\!\Big( -\frac{\pi}{4} e_2 e_4 \Big),
\]
and
\[
\xi_4' = (\rho_2' \otimes \rho_2, \; \rho_2 \otimes \rho_2'),
\]
where $\rho_2, \rho_2'$ are $2$-dimensional representations of $\SL(2, \mathbb{C})$.

We have a chain of sets
\[
\mathcal{E}_{\text{simp}}(\widetilde{G}) \subset \mathcal{E}_{\text{ell}}(\widetilde{G}) \subset \mathcal{E}(\widetilde{G}),
\]
which are finite for local $F$ and infinite for global $F$.

We thus obtain the triality twisted trace formula for $\PGSO(8) \rtimes \theta$.

\begin{theorem}\label{twisted endoscopic data}
If $F$ is a local field or number field, and $\widetilde{G} = \PGSO(8) \rtimes \theta$ with $\theta^3 = 1$, then the triality twisted elliptic endoscopic data of $\widetilde{G}$ are
\[
\big(G_2, 1 \rtimes \theta, \xi\big),\qquad 
\big(\SO(4), s_4' \rtimes \theta, \xi_4'\big),\qquad 
\big(\SL(3), s_3' = \exp\!\Big( \frac{2\pi}{3} \cdot \frac{e_2 e_6 - e_3 e_7}{2} \Big) \rtimes \theta, \Ad\big).
\]
\end{theorem}

If $F$ is a local or global field, then $\mathcal{E}_{\text{ell}}(G_2)$ contains $G_2$, $\PGL(3)$, and $\SO(4)$. The dual group of $\PGL(3)$ is $\SL(3, \mathbb{C})$, which embeds into $G_2(\mathbb{C})$ by
\[
\xi_3: \SL(3, \mathbb{C}) \longrightarrow G_2(\mathbb{C}), \qquad
x \longmapsto \diag(x, 1, {}^t x^{-1}), \quad x \in \SL(3, \mathbb{C}).
\]
We take $s_3 = (-1, -1, 1, -1, -1, 1, 1)$; then
\[
\xi_3(\SL(3, \mathbb{C})) = \big( C_{G_2(\mathbb{C})}(s_3) \big)^\circ.
\]
Thus $(\PGL(3), s_3, \xi_3)$ is an elliptic endoscopic datum for $G_2$.

The dual group of $\SO(4)$ is $\SO(4, \mathbb{C})$; however, its embedding into $G_2$ is not as straightforward as it might seem. Let $\mathbb{O}_{\mathbb{C}}$ be the complex octonions, and $G_2(\mathbb{C}) = \Aut(\mathbb{O}_{\mathbb{C}})$. Fix a quaternion subalgebra $\mathbb{H}_{\mathbb{C}} \subset \mathbb{O}_{\mathbb{C}}$ spanned by $\{1, i, j, k\}$. Then
\[
\mathbb{O}_{\mathbb{C}} = \mathbb{H}_{\mathbb{C}} \oplus \mathbb{H}_{\mathbb{C}} \ell,
\]
where $\ell$ is a perpendicular imaginary unit. Then $\SO(4, \mathbb{C})$ acts on $\mathbb{O}_{\mathbb{C}}$ by
\[
(x_1, x_2) \cdot (v, X) = (x_1 v \overline{x}_1,\; x_1 X \overline{x}_2),
\]
where
\[
v \in \Im(\mathbb{H}_{\mathbb{C}}) \cong \mathbb{C}^3, \qquad 
X \in \mathbb{H}_{\mathbb{C}} \ell \cong \mathbb{C}^4,
\]
and $(x_1, x_2) \in \SL(2, \mathbb{C})_l \times \SL(2, \mathbb{C})_s$, after identifying unit quaternions with $\SL(2, \mathbb{C})$. The kernel is $\mu_2 = \{\pm(I, I)\}$, giving the quotient
\[
\big( \SL(2, \mathbb{C})_l \times \SL(2, \mathbb{C})_s \big) / \mu_2 \cong \SO(4, \mathbb{C}).
\]
The first factor $\SL(2, \mathbb{C})_l$ is associated with the long root $\beta$, and the second factor $\SL(2, \mathbb{C})_s$ with the short root $2\alpha + \beta$, where $\alpha$ is the short root of $G_2$. This action preserves the complex $G_2$ $3$-form and the octonion product; hence $\SO(4, \mathbb{C})$ embeds into $G_2(\mathbb{C})$. We denote this embedding by $\xi_4$.

If we take $s_4 = (1, 1, 1, -1, -1, -1, -1)$, then
\[
\big( C_{G_2(\mathbb{C})}(s_4) \big)^\circ \cong \SO(4, \mathbb{C}).
\]
Thus $(\SO(4), s_4, \xi_4)$ is an elliptic endoscopic datum for $G_2$.

However, $\PGL(3)$ and $\SO(4)$ are not compatible with the triality action and do not arise as twisted endoscopic data for $\widetilde{G}$. We have thus obtained the explicit elliptic endoscopic data for $G_2$.

\begin{theorem}\label{elliptic endoscopic data}
If $F$ is a local field or number field, and $G = G_2$, then the elliptic endoscopic data of $G$ are
\[
(G_2, 1, \Id), \qquad 
(\PGL(3), s_3, \xi_3), \qquad 
(\SO(4), s_4, \xi_4),
\]
where $s_3 = (-1, -1, 1, -1, -1, 1, 1)$ and $s_4 = (1, 1, 1, -1, -1, -1, -1)$.
\end{theorem}

\begin{remark}
To classify automorphic representations of $G_2$, we will compare the triality twisted trace formula for $\PGSO(8)$ with the standard trace formula for $G_2$. The present paper lays the foundation by establishing the twisted trace formula for $\PGSO(8)$.
\end{remark}

Our ultimate goal is to obtain an endoscopic classification of automorphic representations of $G_2$ via comparison of trace formulas. In this paper, we focus on establishing the triality twisted trace formula for $\PGSO(8)$.

  \section{Trace formula of $\G_{2}$}
Let $F$ be a number field with adele ring $\mathbb{A}$, and let $G$ be a connected reductive algebraic group over $F$. Then $G(\mathbb{A})$ is a locally compact topological group containing $G(F)$ as a discrete subgroup. Let $\mathcal{H}(G)$ denote the Hecke algebra of smooth, compactly supported, complex-valued functions on $G(\mathbb{A})$ that are $K$-finite for some maximal compact subgroup $K \subset G(\mathbb{A})$.

The regular representation $R$ of $G(\mathbb{A})$ on $L^2(G(F)\backslash G(\mathbb{A}))$ is defined by
\[
(R(y)\phi)(x) = \phi(xy), \qquad x, y \in G(\mathbb{A}), \ \phi \in L^2(G(F)\backslash G(\mathbb{A})).
\]
For $f \in \mathcal{H}(G)$, the convolution operator $R(f) = \int_{G(\mathbb{A})} f(y) R(y) \, dy$ has kernel
\[
K(x, y) = \sum_{\gamma \in G(F)} f(x^{-1} \gamma y).
\]
This kernel admits two natural decompositions
\[
K(x, y) = \sum_{\mathfrak{o} \in \mathcal{O}} K_{\mathfrak{o}}(x, y) = \sum_{\chi \in \Pi} K_{\chi}(x, y),
\]
where $\mathcal{O}$ is the set of conjugacy classes in $G(F)$, and $\Pi$ is a set of spectral data.

To handle divergence issues, Arthur introduced truncated kernels $K_{\mathfrak{o}}^T(x, f)$ and $K_{\chi}^T(x, f)$ depending on a parameter $T$ in a positive Weyl chamber. This yields the identity
\begin{equation} \label{kernel truncation}
\sum_{\mathfrak{o} \in \mathcal{O}} K_{\mathfrak{o}}^T(x, f) = \sum_{\chi \in \Pi} K_{\chi}^T(x, f).
\end{equation}
Integrating both sides gives the weighted trace formula
\begin{equation} \label{weighted trace formula}
J(f) = \sum_{\mathfrak{o} \in \mathcal{O}} J_{\mathfrak{o}}(f) = \sum_{\chi \in \Pi} J_{\chi}(f).
\end{equation}

Arthur \cite[Theorem 8.1]{A7} defined the weighted orbital distributions $J_M(\gamma, f)$ for Levi subgroups $M$ and $\gamma \in M(F)$, leading to the expansion
\begin{equation} \label{weighted orbit formula}
J_{\mathfrak{o}}(f) = \sum_{M \in \mathcal{L}} |W_0^M| \, |W_0^G|^{-1} \sum_{\gamma \in (M(F) \cap \mathfrak{o})_{M,S}} a^M(S, \gamma) \, J_M(\gamma, f),
\end{equation}
where $\mathcal{L} = \mathcal{L}(M_0)$ is the set of Levi subgroups containing a fixed minimal Levi $M_0$, and $W(M) = \operatorname{Norm}_{G}(A_M)/M$ denotes the relative Weyl group.

To obtain invariant distributions, define inductively
\[
I_M(\gamma, f) = J_M(\gamma, f) - \sum_{\substack{L \in \mathcal{L}(M) \\ L \neq G}} \hat{I}_M^L(\gamma, \phi_L(f)),
\]
where $\phi_L$ transfers functions from $G(F_S)$ to $L(F_S)$. Then the invariant geometric expansion becomes
\[
I(f) = \sum_{M \in \mathcal{L}} |W_0^M| \, |W_0^G|^{-1} \sum_{\gamma \in (M(F))_{M,S}} a^M(S, \gamma) \, I_M(\gamma, f).
\]

The spectral side admits a parallel decomposition
\[
I(f) = \sum_{t \geq 0} I_t(f) = \sum_{t \geq 0} \sum_{M \in \mathcal{L}} |W_0^M| \, |W_0^G|^{-1} \int_{\Pi(M, t)} a^M(\pi) \, I_M(\pi, f) \, d\pi,
\]
where $I_t(f)$ sums over classes $\chi$ with $\|v_{\pi,I}\| = t$ for $v_\pi$ the infinitesimal character, and $v_{\pi,I}$ is the imaginary part of $v_\pi$.

To stabilize the trace formula, define
\[
S(f) = I(f) - \sum_{\substack{G' \in \mathcal{E}_{\text{ell}}(G) \\ G' \neq G}} \iota(G, G') \, \hat{S}^{G'}(f'),
\]
where $f \mapsto f'$ is the global transfer map, and
\[
\iota(G, G') = k(G, G') \, |\overline{Z}(\widehat{G}')^\Gamma|^{-1} \, |\operatorname{Out}_G(G')|^{-1},
\]
with
\[
k(G, G') = |\ker^1(F, Z(\widehat{G}^0))|^{-1} \, |\ker^1(F, Z(\widehat{G}'))|.
\]
Arthur \cite{A6} proved, under the weighted fundamental lemma (established by Chaudouard–Laumon \cite{CL1,CL2}), that $S(f)$ is stable and admits a decomposition
\[
S(f) = \sum_{M \in \mathcal{L}} |W_0^M| \, |W_0^G|^{-1} \sum_{\delta \in \Delta(M, V)} b^M(\delta) \, S_M(\delta, f),
\]
where $\Delta(M, V)$ denotes stable conjugacy classes in $M(F_V)$.

The discrete part of the invariant trace formula \cite{A2,A3} expands as
\begin{align*}
I_{\text{disc},t}(f) &= \sum_{M \in \mathcal{L}} |W_0^M| \, |W_0^G|^{-1} \\
&\quad \times \sum_{w \in W(M)_{\text{reg}}} |\det(w - 1)_{\mathfrak{a}^G_M}|^{-1} \, \tr\bigl( M_{P,t}(w, \chi) \, \mathcal{I}_{P,t}(\chi, f) \bigr),
\end{align*}
where $W(M)_{\text{reg}} = \{ w \in W(M) : \det(w - 1)_{\mathfrak{a}^G_M} \neq 0 \}$, and $M_{P,t}(w, \chi)$ is the global intertwining operator
\[
M_{P}(w, \chi_\lambda) = \ell(w) \circ M_{P'|P}(\chi_\lambda), \qquad P' = w^{-1}P,
\]
defined meromorphically in $\lambda \in (\mathfrak{a}^G_M)^*_{\mathbb{C}}$, and unitary on $i(\mathfrak{a}^G_M)^*$.

The stabilized discrete trace is
\[
S_{\text{disc},t}(f) = I_{\text{disc},t}(f) - \sum_{\substack{G' \in \mathcal{E}_{\text{ell}}(G) \\ G' \neq G}} \iota(G, G') \, \hat{S}'_{\text{disc},t}(f'),
\]
and we obtain the full stabilization
\[
I_{\text{disc},t}(f) = \sum_{G' \in \mathcal{E}_{\text{ell}}(G)} \iota(G, G') \, \hat{S}'_{\text{disc},t}(f').
\]

\begin{theorem}
For $G = G_2$ and $f \in \mathcal{H}(G)$, the discrete invariant distribution expands as
\begin{align*}
I_{\text{disc},t}(f) = &\; \tr\bigl( \mathcal{I}_{G,t}(1,f) \bigr) \\
& + \frac{1}{6} \sum_{w \in W(\GL(2)_s)_{\text{reg}}} |\det(w-1)_{\mathbb{R}}|^{-1} \, \tr\bigl( M_{P,t}(w,1) \, \mathcal{I}_{P,t}(1,f) \bigr) \\
& + \frac{1}{6} \sum_{w \in W(\GL(2)_l)_{\text{reg}}} |\det(w-1)_{\mathbb{R}}|^{-1} \, \tr\bigl( M_{P',t}(w,1) \, \mathcal{I}_{P',t}(1,f) \bigr) \\
& + \frac{1}{12} \sum_{w \in W(T)_{\text{reg}}} |\det(w-1)_{\mathbb{R}^2}|^{-1} \, \tr\bigl( M_{B,t}(w,\chi) \, \mathcal{I}_{B,t}(1,f) \bigr),
\end{align*}
where $\GL(2)_s$ and $\GL(2)_l$ denote the Levi subgroups corresponding to the short and long roots, respectively.
\end{theorem}

\begin{theorem}
For $G = G_2$ and $f \in \mathcal{H}(G)$, the discrete invariant distribution stabilizes as
\[
I_{\text{disc},t}(f) = S^{G_2}_{\text{disc},t}(f) + \iota(G,\SO(4)) \, \widehat{S}^{\SO(4)}_{\text{disc},t}(f') + \iota(G,\PGL(3)) \, \widehat{S}^{\PGL(3)}_{\text{disc},t}(f''),
\]
and the full invariant trace formula becomes
\[
I(f) = \sum_{G' \in \mathcal{E}_{\text{ell}}(G)} \iota(G,G') \, \widehat{S}'(f'),
\]
where $\mathcal{E}_{\text{ell}}(G) = \{ G_2, \SO(4), \PGL(3) \}$, and $f', f''$ are the Langlands–Kottwitz–Shelstad transfers of $f$.
\end{theorem}

\section{The Triality Twisted Trace Formula for $\PGSO(8)$}
Let $\widetilde{G} = (G, \theta, \omega)$ be a triple over $F$, where $G$ is connected reductive, $\theta$ is a semisimple automorphism, and $\omega$ is a character on $G(\mathbb{A})$ trivial on $G(F)$. We write $\widetilde{G} = G \rtimes \theta$ for the associated $G$-bitorsor.

The Hecke module $\mathcal{H}(\widetilde{G})$ consists of smooth, compactly supported, $K$-finite functions on $\widetilde{G}$. The conjugation actions are
\begin{align*}
x^{-1}(y \rtimes \theta)x &= x^{-1}y\theta(x) \rtimes \theta, \quad x,y \in G, \\
(x \rtimes \theta)^{-1} y (x \rtimes \theta) &= \theta^{-1}(x^{-1}yx), \quad x,y \in G.
\end{align*}

Define:
\begin{align*}
\mathfrak{a}_{\widetilde{G}} &= \mathfrak{a}_G^\theta = \{H \in \mathfrak{a}_G : \Ad(\theta)H = H\}, \\
A_{\widetilde{G}} &= (A_G^\theta)^0, \\
Z(\widetilde{G}(\mathbb{A})) &= Z(G(\mathbb{A}))^\theta.
\end{align*}

Let $\chi$ be a character on $(\ker \omega \cap Z(\widetilde{G}(\mathbb{A})))/(Z(\widetilde{G}(F)) \cap \ker \omega)$. The equivariant Hecke module $\mathcal{H}(\widetilde{G}, \chi)$ consists of functions satisfying
\[
f(xz) = f(x)\chi(z)^{-1}, \quad z \in \ker \omega \cap Z(\widetilde{G}(\mathbb{A})).
\]

For $M \in \mathcal{L}(M_0)$, define
\begin{align*}
\mathfrak{a}_{G,\theta} &= \mathfrak{a}_G / \{X - \Ad(\theta)X : X \in \mathfrak{a}_G\}, \\
\mathfrak{a}^{\widetilde{G}}_M &= \ker(\mathfrak{a}_M \to \mathfrak{a}_{\widetilde{G}}), \\
W(M) &= W^{\widetilde{G}}(M) = \operatorname{Norm}(A_M, \widetilde{G})/M, \\
W_{\text{reg}}(M) &= \{w \in W(M) : \det(w-1)_{\mathfrak{a}^{\widetilde{G}}_M} \neq 0\}.
\end{align*}

For $P = N_P M \in \mathcal{P}(M)$, let $\mathcal{H}_P$ be the space of left $N_P(\mathbb{A})$-invariant functions $\phi$ on $G(\mathbb{A})$ with $\phi(mk) \in L^2_{\text{disc}}(M(F)\backslash M(\mathbb{A}), \chi) \otimes L^2(K)$. Let $\mathcal{H}_{P,t}(\chi)$ be the invariant subspace under $M_P(w,\chi_\lambda)$.

Define $\widetilde{\mathcal{H}}_{P,t}(\chi)$ as functions $\phi$ on $\widetilde{G}(\mathbb{A})$ such that $\phi(xy) \in \mathcal{H}_{P,t}(\chi)$ for all $y \in \widetilde{G}(\mathbb{A})$. The induced representation is
\[
(\mathcal{I}_P^G(\chi, y)\phi)(x) = \phi(xy)\omega(\theta^{-1}xy), \quad \phi \in \widetilde{\mathcal{H}}_{P,t}(\chi).
\]

For $f \in \mathcal{H}(\widetilde{G}, \chi)$, define
\[
\mathcal{I}_{P,t}(\chi, f) = \int_{\widetilde{G}(\mathbb{A})/(Z(\widetilde{G}(\mathbb{A})) \cap \ker \omega)} f(y) \, \mathcal{I}^G_{P,t}(x, y) \, dy.
\]

The intertwining operator $M_{P,t}(w, \chi_{\lambda})$ is the meromorphic composition
\[
\mathcal{H}_{P,t}(\chi) \xrightarrow{M_{P'|P}(\chi_\lambda)} \mathcal{H}_{P',t}(\chi) \xrightarrow{\ell(w)} \widetilde{\mathcal{H}}_{P,t}(\chi),
\]
where $(\ell(w)\phi)(x) = \phi(w^{-1}x\theta^{-1})$. At $\lambda = 0$, denote this by $M_{P,t}(w, \chi)$.

For $M = G$, the operator $I_G(\chi, f) \circ M_G(\theta, \chi)$ acts on $L^2(G(F)\backslash G(\mathbb{A}), \chi)$ with kernel
\[
K(x, y) = \sum_{\gamma \in \widetilde{G}} f(x^{-1}\gamma y)\omega(y), \quad x,y \in G(F)\backslash G(\mathbb{A}).
\]
The twisted trace formula is concerned with the study of this kernel. In our case, $\omega(y)$ is trivial.

The twisted invariant trace formula $I(f)$ and its discrete part $I_{\text{disc},t}(f)$ can be established as in \cite{LW}. The stabilization takes the form
\begin{equation} \label{twisted trace formula}
I(f) = \sum_{G' \in \mathcal{E}_{\text{ell}}(\widetilde{G})} \iota(\widetilde{G}, G') \, \widehat{S}'(f'),
\end{equation}
and
\begin{equation} \label{twisted trace 1}
I_{\text{disc},t}(f) = \sum_{G' \in \mathcal{E}_{\text{ell}}(\widetilde{G})} \iota(\widetilde{G}, G') \, \widehat{S}'_{\text{disc},t}(f'),
\end{equation}
with coefficients
\[
\iota(\widetilde{G}, G') = |\pi_0(\kappa_{\widetilde{G}})|^{-1} \, k(\widetilde{G}, G') \, |\overline{Z}(\widehat{G}')^\Gamma|^{-1} \, |\operatorname{Out}_{\widetilde{G}}(G')|^{-1},
\]
where $\kappa_{\widetilde{G}} = Z(\widehat{\widetilde{G}}^\Gamma) \cap (Z(\widehat{G}^0)^\Gamma)^0$.

Assume that the twisted weighted fundamental lemma holds. Moeglin and Waldspurger \cite{MW2,MW3} established the two identities (\ref{twisted trace formula}) and (\ref{twisted trace 1}); we thus obtain the following theorem.

\begin{theorem}\label{twisted trace formula of PGSO8}
Let $\widetilde{G} = \mathrm{PGSO}(8) \rtimes \theta$ with $\theta^3 = 1$, and let $f \in \widetilde{\mathcal{H}}(G, \chi)$. Then the triality twisted trace formula is given by
\[
I(f) = \widehat{S}^{G_2}(f') + \tfrac{1}{4} \widehat{S}^{\SO(4)}(f'_4) + \tfrac{1}{3} \widehat{S}^{\SL_3}(f'_3),
\]
and its discrete part satisfies
\begin{equation} \label{invariant trace formula of PGSO}
I_{\text{disc},t}(f) 
= \tr\bigl( \mathcal{I}_{\widetilde{G},t}(f) \bigr)
+ \tfrac{1}{6} \sum_{w \in W_{\text{reg}}(M)} |\det(w-1)_{\mathfrak{a}^{\widetilde{G}}_M}|^{-1} 
\, \tr\bigl( M_{P,t}(w,1) \, \mathcal{I}_{P,t}(1,f) \bigr)
\end{equation}
and
\begin{align*}
I_{\text{disc},t}(f)
&= \widehat{S}^{G_2}_{\mathrm{disc},t}(f') 
   + \tfrac{1}{4} \widehat{S}^{\SO(4)}_{\mathrm{disc},t}(f'_4) 
   + \tfrac{1}{3} \widehat{S}^{\SL_3}_{\mathrm{disc},t}(f'_3),
\end{align*}
where $M = \GL(2)$, and $f', f'_4, f'_3$ denote the Langlands–Kottwitz–Shelstad transfers of $f$.
\end{theorem}

\begin{remark}
The stable distributions $\widehat{S}^{\SL_3}_{\text{disc},t}(f'_3)$ and $\widehat{S}^{\SO(4)}_{\text{disc},t}(f'_4)$ are well-understood, while $\widehat{S}^{G_2}_{\text{disc},t}(f')$ requires further study to obtain explicit formulas.
\end{remark}

 \section{Expansion of discrete invariant distribution} 

 Assume that $F$ is a number field. A hypothetical global Langlands group $L_F$ \cite[\S 2]{L1} is thought to be a locally compact extension
\[
1 \longrightarrow K_F \longrightarrow L_F \longrightarrow W_F \longrightarrow 1
\]
of the Weil group by a compact connected group $K_F$. However, its existence is a very deep problem, and Arthur \cite{A1} constructed an approximation $L_F^*$ of the global Langlands group, which depends on the endoscopic classification of automorphic representations of classical groups.

\subsection{A-parameters}

An Arthur parameter (A-parameter) of $\GL(8)$ is an equivalence class of unitary $8$-dimensional representations
\[
\psi: L_F^* \times \SL(2,\mathbb{C}) \longrightarrow \GL(8,\mathbb{C}).
\]
We denote by $\Psi(\GL(8))$ the set of A-parameters of $\GL(8)$. To any $\psi \in \Psi(\GL(8))$, we attach an $8$-dimensional representation
\[
\phi_\psi(u) = \psi\bigg(u, \begin{pmatrix}
|u|^{\frac12} & 0 \\
0 & |u|^{-\frac12}
\end{pmatrix}\bigg), \qquad u \in L_F^*,
\]
of $L_F^*$. This corresponds to the unique irreducible quotient $\pi_\psi$ established by Moeglin and Waldspurger \cite{MW1}. We then have a bijection from $\Psi(\GL(8))$ to $\mathcal{A}(\GL(8))$ mapping $\psi$ to $\pi_\psi$. Here $\mathcal{A}(\GL(8))$ is the subset of irreducible unitary representations $\pi$ of $\GL(8,\mathbb{A})$ whose restrictions to $\GL(8,\mathbb{A})^1$ are irreducible constituents of the space $L^2(\GL(8,F)\backslash\GL(8,\mathbb{A})^1)$, where
\[
\GL(8,\mathbb{A})^1 = \{ x \in \GL(8,\mathbb{A}) : |\chi(x)| = 1,\ \chi \in X(\GL(8))_F \}
\]
is the closed subgroup of $\GL(8,\mathbb{A})$ and $X(\GL(8))_F$ is the additive group of characters of $\GL(8)$ defined over $F$.

An automorphic representation $\pi$ of $\GL(8)$ is a weakly continuous, irreducible representation of $\GL(8)$. It can be written as a restricted tensor product
\[
\pi = \bigotimes_v' \pi_v
\]
of irreducible representations of the local groups $\GL(F_v)$, almost all of which are unramified \cite{F}. When $\pi_v$ is unramified, the local Langlands correspondence is well known: there is a bijection between the set of unramified representations of $\GL(8,F_v)$ and the set of semisimple $\GL(8,\mathbb{C})$-orbits in $\GL(8,\mathbb{C}) \times W_{F_v}$ that project to the Frobenius generator of $W_{F_v}/I_{F_v}$, where $I_{F_v}$ is the inertia subgroup of the local Weil group $W_{F_v}$.

The automorphic representation $\pi$ of $\GL(8)$ gives rise to a family of semisimple conjugacy classes
\[
c^S(\pi) = \{ c_v(\pi) : v \notin S \}
\]
in $\GL(8,\mathbb{C})$, where $S$ is some finite set of valuations of $F$ outside of which $\GL(8)$ is unramified. We define $c^S(\pi)$ and $(c')^{S'}(\pi)$ to be equivalent if $c_v(\pi) = c'_v(\pi)$ for almost all $v$. Denote by $\mathcal{C}_{\mathrm{aut}}(\GL(8))$ the set of equivalence classes of $c^S(\pi)$. This yields a map $\pi \mapsto c(\pi)$ from the set of automorphic representations of $\GL(8)$ onto $\mathcal{C}_{\mathrm{aut}}(\GL(8))$.

The outer automorphism $\theta: x \mapsto x^\vee$ of $\GL(8)$ acts on $\Psi(\GL(8))$, transforming a parameter $\psi$ to its contragredient $\psi^\vee$. We denote the subset
\[
\widetilde{\Psi}(\GL(8) \rtimes \theta) = \{ \psi \in \Psi(\GL(8)) : \psi^\vee = \psi \}
\]
of self-dual parameters in $\Psi(\GL(8))$.

We denote by $\Psi(\SO(8))$ the set of conjugacy classes of $L$-homomorphisms
\[
\psi': L_F^* \times \SL(2,\mathbb{C}) \longrightarrow \SO(8,\mathbb{C}) \times W_F.
\]
For any $\psi' \in \Psi(\SO(8))$, there exists an A-parameter $\psi \in \widetilde{\Psi}(\GL(8))$ such that
\[
\xi \circ \psi' = \psi,
\]
where $\xi$ is the embedding of $\SO(8,\mathbb{C}) \times W_F$ into $\GL(8,\mathbb{C}) \times W_F$ that is part of the twisted endoscopic datum. Since $\psi$ and $\xi$ are to be regarded as $\GL(8,\mathbb{C})$-conjugacy classes of homomorphisms, there are two $\SO(8,\mathbb{C})$-orbits of homomorphisms $\psi'$ in the class of $\psi$. We write $\widetilde{\Psi}(\SO(8))$ for the set of $\mathrm{O}(8,\mathbb{C})/\SO(8,\mathbb{C})$-orbits $\psi'$ in $\Psi(\SO(8))$. Arthur \cite{A1} constructed the global Langlands correspondence for discrete series of $\SO(8)$ (modulo $\mathrm{O}(8)/\SO(8)$).

We write $\Psi(\PGSO(8))$ for the set of conjugacy classes of $L$-homomorphisms
\[
\psi: L_F^* \times \SL(2,\mathbb{C}) \longrightarrow \Spin(8,\mathbb{C}) \times W_F.
\]

Given $\psi \in \Psi(\PGSO(8))$, we have the following commutative diagram:
\[
\begin{tikzcd}
L_F^* \arrow[r, "\psi"] \arrow[dr, "\psi'"'] & \Spin(8,\mathbb{C}) \arrow[d, "\rho"] \\
                                   & \SO(8,\mathbb{C})
\end{tikzcd}
\]
Thus $\rho \circ \psi = \psi'$, yielding an A-parameter $\psi' \in \Psi(\SO(8))$. The kernel of $\rho$ is $\mathbb{Z}_2$. Given an A-parameter $\psi'$, a natural question is whether there exists a lift $\psi$ such that $\rho \circ \psi = \psi'$.

In this paper, we consider the \emph{triality A-parameter} $(\psi'_1, \psi'_2, \psi'_3)$, where $\psi'_i \in \Psi(\SO(8))$. The composition
\[
L_F^* \xrightarrow{\psi} \Spin(8,\mathbb{C}) \xrightarrow{(\rho, \rho \circ \theta, \rho \circ \theta^2)} \SO(8,\mathbb{C})^3
\]
equals $(\psi'_1, \psi'_2, \psi'_3)$. The automorphism $\theta$ induces an outer automorphism of $\SO(8,\mathbb{C})^3$ such that
\[
\theta(\rho, \rho \circ \theta, \rho \circ \theta^2) = (\theta(\rho), \theta^2(\rho), \rho) = (\rho \circ \theta, \rho \circ \theta^2, \rho).
\]
Hence
\[
(\rho, \rho \circ \theta, \rho \circ \theta^2)(\theta(\psi)) = \theta(\rho, \rho \circ \theta, \rho \circ \theta^2)(\psi).
\]
The symmetric group $S_3$ acts naturally on $\SO(8,\mathbb{C})^3$; there exists $\sigma \in S_3$ such that
\[
\sigma(\rho, \rho \circ \theta, \rho \circ \theta^2) = (\rho \circ \theta^2, \rho, \rho \circ \theta).
\]
Set $\tau = \sigma \theta^*$; then
\[
\tau(\rho, \rho \circ \theta, \rho \circ \theta^2) = (\rho, \rho \circ \theta, \rho \circ \theta^2),
\]
so $\tau(\psi'_1, \psi'_2, \psi'_3) = (\psi'_1, \psi'_2, \psi'_3)$ and $\tau^3 = 1$. The triality A-parameter $(\psi'_1, \psi'_2, \psi'_3)$ is therefore $\tau$-stable. We denote by $\Psi^\tau(\PGSO(8))$ the set of conjugacy classes of triality A-parameters.

We write $\widetilde{\Psi}(\PGSO(8))$ for the set of $\theta$-stable A-parameters in $\Psi(\PGSO(8))$. For $\psi \in \Psi(\PGSO(8))$, we have an irreducible decomposition
\[
\psi = \ell_1 \psi_1 \oplus \ell_2 \psi_2 \oplus \cdots \oplus \ell_k \psi_k
\]
for inequivalent irreducible representations $\psi_i$ of $L_F^* \times \SL(2,\mathbb{C})$ of degree $n_i$, with multiplicities $\ell_i$ such that
\[
\sum_{i=1}^k \ell_i n_i = 8.
\]
The representation $\psi$ is $\theta$-stable if and only if there exists a cyclic permutation $\sigma$ of order $3$ on the indices such that for all $i$,
\[
\theta(\psi_i) = \psi_{\sigma(i)}, \qquad \theta^2(\psi_i) = \psi_{\sigma^2(i)},
\]
and $\ell_i = \ell_{\sigma(i)} = \ell_{\sigma^2(i)}$. We say that $\psi$ is \emph{square integrable} if $\ell_i = 1$ for all $i$, and \emph{elliptic} if $\ell_i \leq 2$ for all $i$.

Assume $\psi$ is square integrable and $\theta$-stable. Then there are two types of components $\psi_i$: $\theta$-stable and unstable. If $\psi_i$ is an irreducible $\theta$-stable representation, we can write
\[
\theta(\psi_i)((w,y)) = \theta(\psi_i((w,y))) = x_i \psi_i((w,y)) x_i^{-1}, \qquad (w,y) \in L_F^* \times \SL(2,\mathbb{C}),
\]
for a fixed element $x_i \in \SO(n_i,\mathbb{C})$. Applying $\theta$ twice yields
\[
\theta^2(\psi_i((w,y))) = \theta(x_i) \theta(\psi_i((w,y))) \theta(x_i)^{-1}
\]
and
\[
\psi_i((w,y)) = \theta^2(x_i) \theta(x_i) x_i \psi_i((w,y)) x_i^{-1} \theta(x_i)^{-1} \theta^2(x_i)^{-1}.
\]
Since $\psi_i$ is irreducible, the product $\theta^2(x_i) \theta(x_i) x_i$ is a scalar matrix. If $x_i = I_{n_i}$, then $\theta(\psi_i) = \psi_i$ and $\operatorname{Im} \psi_i \subseteq G_2(\mathbb{C})$. As $G_2 \subset \SO(7,\mathbb{C})$, we have $n_i \leq 7$. In particular, if $\psi$ is a $\theta$-stable, $7$-dimensional irreducible representation of $L_F^* \times \SL(2,\mathbb{C})$, then $\operatorname{Im} \psi \subseteq G_2$.

If $x_i = \omega I_{n_i}$, where $\omega$ is a primitive cube root of unity, then $\operatorname{Im}(\psi_i) \subseteq \Ad(\PGL(3,\mathbb{C}))$. If $\psi$ is a $\theta$-stable, $8$-dimensional irreducible representation of $L_F^* \times \SL(2,\mathbb{C})$, then $\operatorname{Im} \psi \subseteq \Ad(\PGL(3,\mathbb{C}))$.

If $\psi_i$ is not $\theta$-stable, it will appear in the decomposition of $\xi'_4 \circ \psi'$ for some $\psi' \in \Psi(\SO(4))$. Alternatively, $\psi_i \in \Psi(\GL(2))$ and there exist A-parameters $\psi_j$ and $\psi_m$, with $j,m \neq i$, satisfying
\[
\psi_j = \theta(\psi_i), \qquad \psi_m = \theta(\psi_j).
\]
If $\psi' \in \Psi(\SO(4))$ is irreducible, then $\psi = \xi_4' \circ \psi' \cong 2\psi'$ is an elliptic A-parameter but not square integrable. Similarly, if $\psi = \xi_4' \circ \psi'$, then $\psi$ is not square integrable.

Suppose $n_i \leq 2$ and $\psi_1, \psi_2$ are distinct irreducible $2$-dimensional representations such that $\theta(\psi_1) \neq \psi_1$, while $\theta(\psi_2)$ is self-dual reducible of dimension $2$ and $\theta(\psi_2) = \psi_2$. Then
\[
\psi = \psi_1 \oplus \theta(\psi_1) \oplus \theta^2(\psi_1) \oplus \psi_2
\]
satisfies $\operatorname{Im} \psi \subseteq \xi_2(\GL(2,\mathbb{C}))$, and $\psi$ is not an elliptic A-parameter.

Thus, if $\psi_i$ is $\theta$-stable and $\psi \in \widetilde{\Psi}_2(\PGSO(8))$, then the image of $\psi$ is contained in $G_2(\mathbb{C})$ or $\Ad(\PGL(3,\mathbb{C}))$.

\begin{remark}
A new phenomenon emerges. Although $\psi'$ itself is not $\theta$-stable, the composition $\xi'_4 \circ \psi'$ is $\theta$-stable and contributes to the $\theta$-stable twisted trace formula. Indeed,
\[
(\rho, \rho \circ \theta, \rho \circ \theta^2)(\xi'_4 \circ \psi') = (\rho \circ \psi', \theta(\rho) \circ \psi', \theta^2(\rho) \circ \psi').
\]
While $\rho \circ \psi'$ is not $\theta$-stable, the triple $(\rho \circ \psi', \theta(\rho) \circ \psi', \theta^2(\rho) \circ \psi')$ is.
\end{remark}

We say a $\theta$-stable A-parameter $\psi$ is \emph{stable} if every irreducible component $\psi_i$ of $\psi$ is $\theta$-stable, and \emph{semi-stable} if there exists an irreducible component $\psi_i$ that is not $\theta$-stable. We then obtain the following lemma.

\begin{lemma} \label{twisted theta stable A-parameter}
If $\psi \in \widetilde{\Psi}_2(\PGSO(8))$ is $\theta$-stable, then there exists a unique A-parameter $\psi' \in \Psi_2(\SO(8))$ such that $\rho \circ \psi = \psi'$. If $\psi$ is a $\theta$-stable elliptic A-parameter which is not square integrable, then $\psi$ is semi-stable and there exists a unique A-parameter $\psi_1 \in \Psi_2(\SO(4))$ such that $\xi'_4(\psi_1) = \psi$. Moreover, $\psi_1$ uniquely determines the $\langle \theta \rangle$-orbit $(\rho \circ \psi, \rho \circ \theta(\psi), \rho \circ \theta^2(\psi))$.
\end{lemma}

   \subsection{Representations}
   Since the automorphism $\theta$ preserves the center $Z(\Spin(8,F))$, it descends to an automorphism of the adjoint group $\PGSO(8,F)$, still write for $\theta$. The projection 
   \[\lambda_{1}:\SO(8,F)\longrightarrow \PGSO(8,F)\]
  gives three maps by $\theta$, that is $(\lambda_{1},\theta\circ \lambda_{1},\theta^{2}\circ \lambda_{1})$.
    We then obtain an epimorphism from $(\SO(8,F))^{3}$ to $\PGSO(8,F)$. 

  An automorphism representation $\pi$ of $\PGSO(8)$ is a restricted tensor product
  \[\pi=\mathop{\widetilde{\otimes}}_{v}\pi_{v},\]
  where $\pi_{v}$ is an irreducible representation of $\PGSO(8)$ that is unramified for almost all $v$. We recall that $\pi_{v}$ is unramified if $v$ is nonarchimedean, and $\pi_{v}$ contains the trivial representation of the hyperspecial maximal compact subgroup $\PGSO(8,\mathfrak{o}_{v})$ of integral points in $\PGSO(8,F_{v})$. There are an isometric isomorphism
  \[L^{2}(\PGSO(8,F)\backslash\PGSO(8,\mathbb{A}))\xrightarrow{\sim}L^{2}(\GSO(8,F)\backslash\GSO(8,\mathbb{A}),1).\]
   Then we can identity the automorphisms $\pi$ of $\PGSO(8)$ with the automorphisms $\pi^{\prime}$ of $\GSO(8)$ which satisfy $\pi^{\prime}|Z(GSO(8))=1$.  

    Assume $F$ is local field, if $\pi^{\prime}$ is an irreducible admissible representation of $\GSO(8,F)$, then the restriction of $\pi^{\prime}$ to $\SO(8,F)$ is a direct sum of finitely many irreducible admissible representations. 
    \begin{theorem}(Adler and Prasad\cite{AP})
   If $\pi^{\prime}$ is an irreducible admissible
representation of $\GSO(8,F)$, then the restriction of $\pi^{\prime}$ to $\SO(8,F)$ is multiplicity free.
    \end{theorem}
  If $\pi$ is an irreducible admissible representation of $\SO(8,F)$, then there exists a unique irreducible admissible representation $\pi^{\prime}$ of $\GSO(8,F)$ up to twisting by $\Hom(\GSO(8,F)\backslash SO(8,F),\mathbb{C}^{\times})$, such that it contains $\pi$ in its restriction to $\SO(8,F)$. 
   If $\pi^{\prime}$ is an irreducible admissible representation of $\GSO(8,F)$, we denote by
  \[X(\pi^{\prime})=\{\omega\in(\GSO(8,F)/Z(\GSO(8,F))\SO(8,F))^{*}:\pi^{\prime}\otimes\omega\cong\pi^{\prime} \}.\]
  If $\pi$ is an irreducible admissible representation of $\SO(8,F)$, we then denote by
  \[\pi^{\GSO(8)}=\{g\in\GSO(8,F):\pi^{g}\cong\pi\}\]
  \begin{theorem}(XU\cite{X})
      Suppose $\pi^{\prime}$ is an irreducible admissible representation of $\GSO(8,F)$, $\pi$ is contained in its restriction to $\SO(8,F)$, and $\omega\in(\GSO(8,F)/Z(\GSO(8,F))\SO(8,F))^{*}$, then
  $\omega$ is in $X(\pi^{\prime})$ if and only if $\omega$ is trivial on $\pi^{\GSO(8)}$. Moreover, the restriction of $\pi^{\prime}$ contains $|X(\pi^{\prime})|$  irreducible admissible representations of $\SO(8)$.  
   If $\pi^{\prime}$ is an irreducible admissible unitary representation of $\GSO(8,F)$, then $\pi^{\prime}$ is a tempered representation of $\GSO(8,F)$ if and only if  its restriction to $\SO(8,F)$ is a tempered representation. The same is true for the essentially discrete series representation. 
  
   \end{theorem}
  
  We note that the out automorphism of $\GSO(8)$ is of order two, thus Bin Xu do not study the triality case. Now we assume that $F$ is the number field. Given an automorphic representation $\pi$ of $\PGSO(8)$, then we obtain an automorphism $\pi^{\prime}$ of $\GSO(8)$, its restriction to
  $\SO(8)$ is an automorphic representation $\pi_{1}$.
  However, given an automorphism representation $\pi_{1}$ of $\SO(8)$, we cannot determine a unique automorphism representation of $\PGSO(8)$ by $\pi_{1}$.

  Thus, given an automorphic representation $\pi$ of $\PGSO(8)$,
  we can define three automorphic representations
  \[\pi_{\lambda_{1}}(x)=:\pi(\lambda_{1}(x))|_{L^{2}(\SO(8,F)\backslash\SO(8,\mathbb{A}))},\]
  \[\theta(\pi_{\lambda_{1}})(x)=:\pi\big(\theta(\lambda_{1}(x))\big)|_{L^{2}(\SO(8,F)\backslash\SO(8,\mathbb{A}))},\]
and 
\[\theta^{2}(\pi_{\lambda_{1}})(x)=:\pi\big(\theta^{2}(\lambda_{1}(x))\big)|_{L^{2}(\SO(8,F)\backslash\SO(8,\mathbb{A}))}.\]

\begin{definition}
    The triplets $(\pi_{1},\pi_{2},\pi_{3})$ formed by the automorphic representations of $\SO(8)$ is called for triality automorphic representation if there exists an automorphic representation $\pi$ of $\PGSO(8)$, such that 
    \[(\pi_{1},\pi_{2},\pi_{3})=(\pi_{\lambda_{1}},\theta(\pi_{\lambda_{1}}),\theta^{2}(\pi_{\lambda_{1}})).\]
\end{definition}

  We denote $\mathcal{A}^{\tau}(\PGSO(8))$ by the set of the conjugacy classes of the triality automorphic representation. We also denote $\mathcal{A}(\PGSO(8))$ by the set of the automorphic representations of $\PGSO(8)$. Thus we can identify the set $\mathcal{A}(\PGSO(8))$ with the set $\mathcal{A}^{\tau}(\PGSO(8))$. 
  Since $\PGSO(8,F)\backslash\PGSO(8,\mathbb{A})$ has finite volume, there is  
  \[L^{2}_{\cusp}(\PGSO(8,F)\backslash\PGSO(8,\mathbb{A}))\subset L^{2}_{\disc}(\PGSO(8,F)\backslash\PGSO(8,\mathbb{A}))\]
  of embedded, right $\PGSO(8,\mathbb{A})$-invariant Hilbert spaces. Here the space $L^{2}_{\cusp}(\PGSO(8,F)\backslash\PGSO(8,\mathbb{A}))$ of cuspidal functions in \[L^{2}(\PGSO(8,F)\backslash\PGSO(8,\mathbb{A}))\] is contained in the subspace $L^{2}_{\disc}(\PGSO(8,F)\backslash\PGSO(8,\mathbb{A}))$ that decomposes under the action of $\PGSO(8,\mathbb{A})$ into a direct sum of irreducible representations. We then have a corresponding subsets of irreducible automorphic representations
  \[\mathcal{A}_{\cusp}(\PGSO(8))\subset\mathcal{A}_{2}(\PGSO(8)),\]
  where $\mathcal{A}_{\cusp}(\PGSO(8))$ and $\mathcal{A}_{2}(\PGSO(8))$ are the subsets of irreducible unitary representations $\pi$ of $\PGSO(\mathbb{A})$ of the spaces $L^{2}_{\cusp}$ and $L^{2}_{\disc}$ respectively.
 We denote $\widetilde{\mathcal{A}}_{2}(\PGSO(8))$ by the set of $\theta$-orbit automorphic representations in $\mathcal{A}_{2}(\PGSO(8))$, which are derived from the discrete automorphic representations of the simple endoscopic data $(\G_{2},1\rtimes\theta,\xi)$ or $(\SL(3),s_{3}^{\prime}\rtimes\theta,\Ad)$.
 
     If $\pi\in \widetilde{\mathcal{A}}_{\cusp}(\PGSO(8))$ is a $\theta$-stable cuspidal generic representation, and the standard $L$-function $L(s,\pi,\std)$ has no pole at $s=1$, then $\pi=\pi_{\Ad}$ is the lift of a cuspidal automorphic representation of $\SL(3)$. If $\pi$ is an irreducible $\theta$-stable, generic automorphic representation, and the standard $L$-function $L(s,\pi,\std)$ has a pole at $s=1$, then $\pi=\pi_{\xi}$ is the lift of a cuspidal automorphic representation of $\G_{2}$ ( which means that $\pi$ restricts to $\G_{2}$ is a cuspidal representation).
     
     If $\pi\in\widetilde{\mathcal{A}}_{2}(\PGSO(8))$, then it correspondences to a discrete automorphism representation of $(\pi^{\prime},\pi^{\prime},\pi^{\prime})\in\mathcal{A}^{\tau}(\PGSO(8))$, where $\pi^{\prime}$ is a discrete automorphic representation of $\SO(8)$ and $\pi^{\prime}=\pi_{\lambda_{1}}=\theta(\pi_{\lambda_{1}})$.    
  If $\pi\in\mathcal{A}(\PGSO(8))$, and $\pi$ is the irreducible component of the induced representation 
  \[I_{P}(\pi_{1}\otimes\pi_{2}\dots\otimes\pi_{r}\otimes\underbrace{\pi_{r+1}\otimes\pi_{r+1}}_{2}\otimes\dots\otimes\underbrace{\pi_{t}\otimes\pi_{t}}_{2}), \qquad 0\leq r\leq t.\]
  $\pi$ is called an elliptic representation. We denote $\mathcal{A}_{\ellip}(\PGSO(8))$ by the set of conjugacy classes of the elliptic representations in $\mathcal{A}(\PGSO(8))$, and $\widetilde{\mathcal{A}}_{\ellip}(\PGSO(8))$ by the set of $\theta$-orbit automorphic representations in $\mathcal{A}_{\ellip}(\PGSO(8))$. 
   If $\pi$ lies in $\widetilde{\mathcal{A}}_{\ellip}(\PGSO(8))$, then it exists a representation \[(\pi_{\lambda_{1}},\theta(\pi_{\lambda_{1}}),\theta^{2}(\pi_{\lambda_{1}}))\in \mathcal{A}^{\tau}(\PGSO(8)),\] which defined by $\pi$. If $\pi_{\lambda_{1}}=\theta(\pi_{\lambda_{1}})$, then $\pi$ a lifting of automorphism representation of the twisted endoscopic group $\G_{2}$ or $\SL(3)$. If $\pi_{\lambda_{1}}\neq\theta(\pi_{\lambda_{1}})$, then $\pi$ is a lifting of the discrete automorphic representation of the twisted endoscopic group $\SO(4)$. 
  
If $\pi$ lies in $\widetilde{\mathcal{A}}(\PGSO(8))$, then $\pi$ is a lifting of the automorphism representation of the twisted Levi subgroup $\GL(2)$ or the twisted endoscopic group $\SO(4)$. Then we have the following lemma

\begin{lemma} \label{twisted classfy the representation}
    There is a bijective mapping $\pi\longmapsto(\pi_{\lambda_{1}},\theta(\lambda_{1}),\theta^{2}(\lambda_{1}))$ from $\widetilde{\mathcal{A}}(\PGSO(8))$ onto $\widetilde{\mathcal{A}}^{\tau}(\PGSO(A))$.
    \begin{itemize}
        \item If $\pi_{\lambda_{1}}=\theta(\pi_{\lambda_{1}})$, then $\pi_{\lambda_{1}}$ uniquely determined the automorphic representation $\pi$, which is a lifting of the automorphic representation of the simple endoscopic group $\G_{2}$ or $\SL(3)$.
        \item If $\pi_{\lambda_{1}}\neq\theta(\pi_{\lambda_{1}})$, then the $\langle\theta\rangle$-orbit $(\pi_{\lambda_{1}},\theta(\pi_{\lambda_{1}}),\theta^{2}(\pi_{\lambda_{1}}))$ uniquely determined the automorphic representation $\pi$, which is a lifting of the automorphic representation of the elliptic endoscopic group $\SO(4)$.
    \end{itemize} 
\end{lemma}

Arthur \cite{A1} has given the endoscopic classification of automorphic representations of $\SO(8)$.

\begin{theorem}(Arthur\cite[theorem 1.5.2]{A1}) Assume that $F$ is global. Then there is an $\widetilde{\mathcal{H}}(\SO(8))$-module isomorphism
\[L^{2}_{\disc}\big(\SO(8,F)\backslash\SO(8,\mathbb{A})\big)\cong\mathop{\bigoplus}_{\psi\in\widetilde{\Psi}_{2}(\SO(8))}\mathop{\bigoplus_{\pi\in\widetilde{\Pi}_{\psi}(\epsilon_{\psi})}}m_{\psi} \pi,\]
where $m_{\psi}$ equals $1$ or $2$, while 
\[\epsilon_{\psi}:\mathcal{S}_{\psi}\longrightarrow\{\pm1\}\]
is a linear character defined explicitly in terms of symplectic $\epsilon$-factors, and $\widetilde{\Pi}_{\psi}(\epsilon_{\psi})$ is the subset of representations $\pi$ in the global packet $\widetilde{\Pi}_{\psi}$ such that the character $\langle \cdot,\pi\rangle$ on $\mathcal{S}_{\psi}$ equals $\epsilon_{\psi}$.
\end{theorem}

  Here the global packet
  \[ \widetilde{\Pi}_{\psi}=\{\mathop{\otimes}_{v}\pi_{v}:\pi_{v}\in\widetilde{\Pi}_{\psi_{v}},\langle\cdot,\pi_{v}\rangle=1 \quad \text{for almost all}\ v\}\]
of (orbits of )representations of $\SO(8,\mathbb{A})$ to any $\psi\in \widetilde{\Psi}(\SO(8))$. The automorphic representation $\pi=\mathop{\otimes}_{v}\pi_{v}$ in the global 
packet determines a character
  \[\langle x,\pi\rangle=\mathop{\Pi}_{v}\langle x_{v},\pi_{v}\rangle, \qquad x\in\mathcal{S}_{\psi},\]
on the global quotient $\mathcal{S}_{\psi}$. If $\psi\in \widetilde{\Psi}_{2}(\SO(8))$, then 
\[\psi=\psi_{1}\oplus\cdots\oplus\psi_{2}.\]
 $m_{\psi}$ equals $1$, unless the rank $n_{i}$ of each of the components $\psi_{i}$ of $\psi$ is even, and $m_{\psi}$ equals $2$. 
  Thus there is a unique correspondence 
  \begin{align*}
      \mathcal{L}_{\SO(8)}:\widetilde{\Psi}_{2}(\SO(8)&)\longrightarrow\widetilde{\mathcal{A}}_{\disc}
(\SO(8))  \\ 
&\psi\longmapsto\widetilde{\Pi}_{\psi}(\varepsilon_{\psi}) \end{align*}
 such that 
   \[f^{\SO(8)}(\psi)=\sum_{\pi\in\widetilde{\Pi}_{\psi}(\varepsilon_{\psi})}\varepsilon_{\psi}(s_{\psi})\tr\pi(f)\]
   is stable character. In particular, if $\psi=\phi$ is generic parameter, the mean is that the restriction of A-parameter $\psi$ to the factor $\SL(2,\mathbb{C})$ is trivial, then $\varepsilon_{\phi}=1$, and if $\phi_{1}\neq\phi_{2}$, then $\widetilde{\Pi}_{\phi_{1}}\cap\widetilde{\Pi}_{\phi_{2}}=\varnothing$. However, 
   \[\bigsqcup_{\phi\in\widetilde{\Phi}_{2}(\SO(8))}m_{\phi}\widetilde{\Pi}_{\phi}\subsetneqq \widetilde{\mathcal{A}}_{2}(\SO(8)).\]
   
   If $\psi\in\widetilde{\Psi}(\SO(8))\backslash\widetilde{\Psi}_{2}(\SO(8))$, then there is a Levi subgroup $M\subset \SO(8)$ attached to $\psi$ which is a product of several general linear groups with a group $\SO(2n_{1})$, for some $n_{i}<4$, and there exists a A-parameter $\psi_{M}\in\widetilde{\Psi}_{2}(M)$ satisfying $\lambda\circ \psi_{M}=\psi$, where $\lambda$ is an $L$-embedding $\lambda:^{L}M\hookrightarrow \SO(8,\mathbb{C})$. We have the A-packet of $\psi$ as the family 
   \[\widetilde{\Pi}_{\psi}=\{\pi=\mathcal{I}_{P}(\pi_{M}): \pi_{M}\in\widetilde{\Pi}_{\psi_{M}}\},\]
a finite set that is bijective with $\widetilde{\Pi}_{\psi_{M}}$. We note that the induced representation $\mathcal{I}_{P}(\pi_{M})$ can be reducible. However, the characters of standard representations and irreducible representations are mutually expressible. Thus we denote the set $\widetilde{\mathcal{A}}^{\prime}(\SO(8))$ by the set of standard representations in A-packet $\widetilde{\Pi}_{\psi}$ with $\psi\in\widetilde{\Psi}(\SO(8))$. We similarly denote $\widetilde{\mathcal{A}}^{\prime}_{\ellip}(\SO(8))$ by the set of standard representations in A-packet $\widetilde{\Pi}_{\psi}$ with $\psi\in\widetilde{\Psi}_{\ellip}(\SO(8))$.

  The correspondence $\mathcal{L}_{\SO(8)}$ is compatible with two chains 
   \begin{align} \label{A-parameter chain}
       \widetilde{\Psi}_{\simp}(\SO(8))\subset\widetilde{\Psi}_{2}(\SO(8))\subset\widetilde{\Psi}_{\ellip}(\SO(8))\subset\widetilde{\Psi}(\SO(8))
      \end{align} 
   and 
   \begin{align} \label{A-automorphic chain}
   \widetilde{\mathcal{A}}_{\cusp}(\SO(8))\subset\widetilde{\mathcal{A}}_{\disc}(\SO(8))\subset\widetilde{\mathcal{A}}^{\prime}_{\ellip}(\SO(8))\subset\widetilde{A}^{\prime}(\SO(8)).
   \end{align}
 The global classification for $\SO(8)$ can be given by the following theorem. 
\begin{theorem} \label{Langlands global correspondence of SO}
   There is a unique correspondence
   \[\mathcal{L}_{\SO(8)}:\widetilde{\Psi}(\SO(8))\longrightarrow\widetilde{\mathcal{A}}^{\prime}(\SO(8))\]
mapping $\psi$ to $\widetilde{\Pi}_{\psi}(\varepsilon_{\psi})$, such that 
\begin{equation} \label{stable character of SO}
    f^{\SO(8)}(\psi)=\sum_{\pi\in\widetilde{\Pi}_{\psi}(\varepsilon_{\psi})}\varepsilon_{\psi}(s_{\psi})\tr\pi(f)
\end{equation}
is stable.
Furthermore, The correspondence $\mathcal{L}_{\SO(8)}$ maps each subset in (\ref{A-parameter chain}) onto its counterpart in (\ref{A-automorphic chain}).
\end{theorem}

 The out automorphism of the $\PGSO(8)$ equals to 
 \[\Aut(\PGSO(8))/\Int(\PGSO(8))=S_{3}.\]
 Set $\mathcal{O}=\Out(\PGSO(8))$, and denote $\Psi^{\mathcal{O}}(\PGSO(8))$ by the set of $\mathcal{O}$-stable A-parameters $\psi$ in $\Psi(\PGSO(8))$. Denote $\Psi_{\ellip}^{\mathcal{O}}(\PGSO(8))$,
  $\Psi_{2}^{\mathcal{O}}(\PGSO(8))$ and $\Psi_{\simp}^{\mathcal{O}}(\SO(8))$ by the set of stable A-parameters of $\psi$ in $\Psi_{\ellip}(\PGSO(8))$, $\Psi_{2}(\PGSO(8))$ and $\Psi_{\simp}(\PGSO(8))$ respectively.

 We write $\mathcal{A}^{\mathcal{O}}(\PGSO(8))$ for the set of 
$\mathcal{O}$-orbit automorphic representations $\pi$ in $A^{\prime}(\PGSO(8))$. Write $\mathcal{A}^{\mathcal{O}}_{\ellip}(\PGSO(8))$, $\mathcal{A}^{\mathcal{O}}_{\disc}(\PGSO(8))$ and $\mathcal{A}^{\mathcal{O}}_{\cusp}(\PGSO(8))$ for the set of $\mathcal{O}$-orbit automorphic representations $\pi$ in $\mathcal{A}_{\disc}(\PGSO(8))$, $\mathcal{A}_{\ellip}(\PGSO(8))$ and $\mathcal{A}_{\cusp}(\PGSO(8))$ respectively.
  We then have the two chains
  \begin{align}\label{A-parameter chain for PGSO}
   \Psi^{\mathcal{O}}_{\simp}(\PGSO(8))\subset\Psi^{\mathcal{O}}_{2}(\PGSO(8))\subset\Psi^{\mathcal{O}}_{\ellip}(\PGSO(8))\subset\Psi^{\mathcal{O}}(\PGSO(8))   
  \end{align}
and 
\begin{align} \label{A-automorphic chain for PGSO}
   \mathcal{A}^{\mathcal{O}}_{\cusp}(\PGSO(8))\subset\mathcal{A}^{\mathcal{O}}_{\disc}(\PGSO(8))\subset\mathcal{A}^{\mathcal{O}}_{\ellip}(\PGSO(8))\subset\mathcal{A}^{\mathcal{O}}(\PGSO(8)).
   \end{align}

We then have the endoscopic classification of $\mathcal{O}$-orbits automorphic representations of $\PGSO(8)$. 
\begin{theorem}  \label{classification of main theorem}
 There is a unique correspondence 
 \[\mathcal{L}_{\PGSO(8)}:\Psi^{\mathcal{O}}(\PGSO(8))\longrightarrow \mathcal{A}^\mathcal{O}(\PGSO(8)),\]
which maps $\psi$ to $\Pi^{\mathcal{O}}_{\psi}$, such that
\begin{equation} \label{stable character of PGSO}
    f^{\PGSO(8)}(\psi)=\sum_{\pi\in\Pi^{\mathcal{O}}_{\psi}(\varepsilon_{\psi})}\varepsilon_{\psi}(s_{\psi})\tr\pi(f)  \qquad f\in\mathcal{H}^{\mathcal{O}}(\PGSO(8))
\end{equation}
is stable. Furthermore, The correspondence $\mathcal{L}_{\PGSO(8)}$ is compatible with the two chains 
(\ref{A-parameter chain for PGSO}) and (\ref{A-automorphic chain for PGSO}), in the sense that it maps each subset in (\ref{A-parameter chain for PGSO}) onto its counterpart in (\ref{A-automorphic chain for PGSO}).

\end{theorem}
\begin{proof}
If $\psi\in\Psi^{\mathcal{O}}_{2}(\PGSO(8))$, then we have 
  \[\psi=\psi_{1}\oplus\psi_{2}\oplus\cdots\oplus\psi_{r},\]
and all $\psi_{i}$ is $\theta$-stable. We consider the following figure \ref{fig:diagram}. 
\begin{figure}[htbp]
\centering
\begin{displaymath}
   \xymatrix@C=1.8cm@R=1.6cm{
   \widetilde{\Psi}_{2}^{\prime}(\SO(8)) 
  \ar[r]^{\mathcal{L}_{\SO(8)}}  &
        \widetilde{\mathcal{A}}^{\prime}_{\disc}(\SO(8)) \ar[d]^{(\lambda',\lambda',\lambda')} \\
        \widetilde{\Psi}^{\tau}_{2}(\PGSO(8))^{\theta}\ar[u]^{P_{1}} 
  \ar[r]^{(\mathcal{L}_{\SO(8)})^{3}}  &
        \widetilde{\mathcal{A}}_{\disc}^{\tau}(\PGSO(8))^{\theta}  \ar[d]^{(\lambda,\lambda\circ\theta,\lambda\circ\theta^{2})^{-1}} \\
        \Psi^{\mathcal{O}}_{2}(\PGSO(8))\ar[u]^{(\rho,\rho\circ\theta,\rho\circ\theta^{2})} \ar@{-->}[r]^{\mathcal{L}_{\PGSO(8)}} & \mathcal{A}^\mathcal{O}_{\disc}(\PGSO(8))   }
\end{displaymath}
\caption{To define the Langlands correspondence}
\label{fig:diagram}
\end{figure}

 Here $P_{1}$ is the projection of the first component of $\widetilde{\Psi}^{\tau}_{2}(\PGSO(8))^{\theta}$, which is the set of $(\psi,\psi,\psi)$ in $\Psi^{\tau}_{2}(\PGSO(8))$ With $\psi\in\widetilde{\Psi}_{2}(\SO(8))$. $\widetilde{\mathcal{A}}_{\disc}^{\tau}(\PGSO(8))$ is the set of $(\pi,\pi,\pi)$ in $\mathcal{A}_{\disc}^{\tau}(\PGSO(8))$ with $\pi\in \widetilde{\mathcal{A}}_{\disc}(\SO(8))$. $\widetilde{\Psi}_{2}^{\prime}(\SO(8))$ is the image of $P_{1}$, and $\widetilde{\mathcal{A}}^{\prime}_{\disc}(\SO(8))$ is the image of the restriction of the correspondence $\mathcal{L}_{\SO(8)}$ on $\widetilde{\Psi}_{2}^{\prime}(\SO(8))$, $\lambda^{\prime}$ is an embedding. Applying the lemma \ref{twisted theta stable A-parameter}, the lemma \ref{twisted classfy the representation}
and the theorem \ref{Langlands global correspondence of SO} to the figure \ref{fig:diagram}, we can define the global Langlands correspondence $\mathcal{L}_{\PGSO(8)}$, such that the figure \ref{fig:diagram} is commutative. We then define the character $f^{\PGSO(8)}(\psi)$ by (\ref{stable character of PGSO}). Since the character $f^{\SO(8)}(\psi)$ is defined by (\ref{stable character of SO}), which is stable, applying the commutative figure \ref{fig:diagram}, we know that $f^{\PGSO(8)}(\psi)$ is stable.
 
 If $\psi\in\Psi^{\mathcal{O}}(\PGSO(8))\backslash\Psi^{\mathcal{O}}_{2}(\PGSO(8))$, and $\psi$ is $\theta$-stable, then we can obtain the Langlands correspondence $\mathcal{L}_{\PGSO(8)}$ as the above case. 
 If $\psi$ is $\theta$ semi-stable, then $\psi$ is from the lifting of $\psi^{\prime}$ in $\widetilde{\Psi}(\SO(4))$,  applying the Lemma (\ref{twisted theta stable A-parameter})
,Lemma (\ref{twisted classfy the representation}) and the classification of automorphic representations of $\SO(4)$, we can obtain the Langlands correspondence and the stable character. Therefore the correspondence $\mathcal{L}_{\PGSO(8)}$ is compatible with the two chains (\ref{A-parameter chain for PGSO}) and (\ref{A-automorphic chain for PGSO}).

\end{proof}

\section{$\psi$-component of the discrete trace formula}

  We introduced the set $\mathcal{C}_{\mathrm{aut}}(\PGSO(8))$ of equivalence classes of families $c$ of semisimple classes in $\Spin(8,\mathbb{C})$. We consider the operator $\mathcal{I}_{P,t}(1,f)$, which is isomorphic to a direct sum of induced representations of the form
\[
\pi = \mathcal{I}_{P}(\pi_{M}), \qquad \|\mu_{\pi,I}\| = t,
\]
in which $\pi_{M}$ is taken from the set of irreducible sub-representations of $L^{2}_{\mathrm{disc}}(M(F)\backslash M(\mathbb{A}))$. Then the class $c(\pi)$ belongs to $\mathcal{C}_{\mathrm{aut}}(\PGSO(8))$.

If $c$ is an arbitrary class in $\mathcal{C}_{\mathrm{aut}}(\PGSO(8))$, we denote
\[
\mathcal{I}_{P,t,c}(1,f) = \bigoplus_{\{\pi : c(\pi) = c\}} \mathcal{I}_{P,\pi}(1,f),
\]
where $\pi = \mathcal{I}_{P}(\pi_{M})$, and $\mathcal{I}_{P,\pi}(1)$ is the sub-representation of $\mathcal{I}_{P,t}(1)$ corresponding to $\pi$. We write $M_{P,t,c}(w,1)$ for the restriction of the operator $M_{P,t}(w,1)$ to the invariant subspace $\mathcal{H}_{P,t,c}(1)$ on which $\mathcal{I}_{P,t,c}(1)$ acts. Then
\[
\operatorname{tr}\bigl( M_{P,t}(w,1) \mathcal{I}_{P,t}(1,f) \bigr) = \sum_{c \in \mathcal{C}_{\mathrm{aut}}(\PGSO(8))} \operatorname{tr}\bigl( M_{P,t,c}(w,1) \mathcal{I}_{P,t,c}(1,f) \bigr).
\]

We have the decomposition
\begin{equation} \label{c-component decomposition}
I_{\mathrm{disc},t}(f) = \sum_{c \in \mathcal{C}_{\mathrm{aut}}(\PGSO(8))} I_{\mathrm{disc},t,c}(f),
\end{equation}
where $I_{\mathrm{disc},t,c}(f)$ is the $c$-component of $I_{\mathrm{disc},t}(f)$.

We now consider the twisted case $\PGSO(8) \rtimes \theta$. We define $\mathcal{C}_{\mathrm{aut}}(\PGSO(8) \rtimes \theta)$ to be the subset of classes $c \in \mathcal{C}_{\mathrm{aut}}(\PGSO(8))$ that are compatible with $\theta$, in the sense that for almost all valuations $v$, the associated conjugacy classes $c_v$ satisfy
\[
\theta_v(c_v) = c_v.
\]
Then the decomposition (\ref{c-component decomposition}) remains valid for $\PGSO(8) \rtimes \theta$. The sum in (\ref{c-component decomposition}) can be taken over a finite set that depends on $f$ through a choice of Hecke type.

The set $\mathcal{C}_{\mathrm{aut}}(\PGSO(8) \rtimes \theta)$ is a direct limit
\[
\mathcal{C}_{\mathrm{aut}}(\PGSO(8) \rtimes \theta) = \varinjlim_{S} \mathcal{C}_{\mathrm{aut}}^{S}(\PGSO(8) \rtimes \theta),
\]
where $\mathcal{C}_{\mathrm{aut}}^{S}(\PGSO(8) \rtimes \theta)$ is the set of families of semisimple conjugacy classes
\[
c^{S} = \{ c_v : v \notin S \},
\]
with $S \supset S_{\infty}$ a sufficiently large finite set of valuations outside of which $\PGSO(8)$ is unramified. An element $c^{S}$ determines a complex-valued character on $\mathcal{H}_{\mathrm{un}}^{S}(\PGSO(8))$ via the Satake isomorphism:
\begin{equation}
h \longmapsto \hat{h}(c^{S}), \qquad h \in \mathcal{H}_{\mathrm{un}}^{S}(\PGSO(8)) = C_{c}^{\infty}(K^{S}\backslash \PGSO(8)/K^{S}),
\end{equation}
where $\hat{h}$ is the Satake transform of $h$. There is an action
\[
f \longmapsto f_{h}, \qquad f \in \mathcal{H}(\PGSO(8), K^{S}), \; h \in \mathcal{H}_{\mathrm{un}}^{S}(\PGSO(8)),
\]
of $\mathcal{H}_{\mathrm{un}}^{S}(\PGSO(8))$ on $\mathcal{H}(\PGSO(8) \rtimes \theta, K^{S})$, the space of functions in $\mathcal{H}(\PGSO(8) \rtimes \theta)$ that are biinvariant under the hyperspecial maximal compact subgroup $K^{S}$. If $\pi$ is an extension to $\PGSO(8) \rtimes \theta$ of an irreducible unitary representation $\pi^{0}$ of $\PGSO(8,\mathbb{A})$, which is unramified outside $S$ and satisfies
\[
\pi(x_{1} x x_{2}) = \pi^{0}(x_{1}) \pi(x) \pi^{0}(x_{2}), \qquad x_{1}, x_{2} \in \PGSO(8,\mathbb{A}),
\]
then
\[
\operatorname{tr}\bigl( \pi(f_{h}) \bigr) = \hat{h}\bigl( c^{S}(\pi) \bigr) \operatorname{tr}\bigl( \pi(f) \bigr).
\]

We can then write
\[
I_{\mathrm{disc},t}(f) = \sum_{c^{S}} I_{\mathrm{disc},t,c^{S}}(f),
\]
where $I_{\mathrm{disc},t,c^{S}}$ is a linear form on $\mathcal{H}(\PGSO(8), K^{S})$ such that
\begin{equation} \label{eigenvector of invariant}
I_{\mathrm{disc},t,c^{S}}(f_{h}) = \hat{h}(c^{S}) I_{\mathrm{disc},t,c^{S}}(f).
\end{equation}
The sum is over a finite subset of $\mathcal{C}_{\mathrm{aut}}^{S}(\PGSO(8))$. If $c$ belongs to $\mathcal{C}_{\mathrm{aut}}(\PGSO(8))$, we can write
\begin{equation} \label{c-compoent expand}
I_{\mathrm{disc},t,c}(f) = \sum_{c^{S} \rightarrow c} I_{\mathrm{disc},t,c^{S}}(f),
\end{equation}
the sum being over the preimage of $c$ in $\mathcal{C}_{\mathrm{aut}}^{S}(\PGSO(8))$.

Since $G_2$ is an elliptic twisted endoscopic datum of $\PGSO(8) \rtimes \theta$, an element $c' \in \mathcal{C}_{\mathrm{aut}}(G_2)$ can transfer to a family $c$ of conjugacy classes for $\PGSO(8) \rtimes \theta$. However, without the principle of functoriality for $\PGSO(8) \rtimes \theta$ and $G_2$, we cannot assert that $c$ belongs to $\mathcal{C}_{\mathrm{aut}}(\PGSO(8) \rtimes \theta)$. We denote
\[
\mathcal{C}_{\mathbb{A}}(\PGSO(8) \rtimes \theta) = \varinjlim_{S} \mathcal{C}_{\mathbb{A}}^{S}(\PGSO(8) \rtimes \theta)
\]
by the set of all equivalence classes of families $c^{S}$ of semisimple conjugacy classes in $\Spin(8,\mathbb{C}) \times W_{F}$ such that for each $v$, $c_v$ projects to a Frobenius class in $W_{F_v}$. There is a canonical map $c' \mapsto c$ from $\mathcal{C}_{\mathbb{A}}(G_2)$ to $\mathcal{C}_{\mathbb{A}}(\PGSO(8) \rtimes \theta)$. We may take $\mathcal{C}_{\mathbb{A}}(\PGSO(8) \rtimes \theta)$ to be the domain of summation in (\ref{c-component decomposition}) if we set $I_{\mathrm{disc},t,c}(f) = 0$ for $c$ in the complement of $\mathcal{C}_{\mathrm{aut}}(\PGSO(8) \rtimes \theta)$ in $\mathcal{C}_{\mathbb{A}}(\PGSO(8) \rtimes \theta)$.

\begin{lemma} \label{stable distribution exist}
\begin{enumerate}
\item[(1)] There is a decomposition
\begin{equation} \label{Stable of G}
S^{G_2}_{\mathrm{disc},t}(f) = \sum_{c \in \mathcal{C}_{\mathbb{A}}(G_2)} S^{G_2}_{\mathrm{disc},t,c}(f), \qquad f \in \mathcal{H}(G_2)
\end{equation}
for stable linear forms $S^{G_2}_{\mathrm{disc},t,c}$ that satisfy the analogues of (\ref{eigenvector of invariant}) and (\ref{c-compoent expand}), and vanish for all $c$ outside a finite subset of $\mathcal{C}_{\mathbb{A}}(G_2)$ which depends on $f$ only through a choice of Hecke type.

\item[(2)] For any $c \in \mathcal{C}_{\mathbb{A}}(\PGSO(8) \rtimes \theta)$, there is a decomposition
\begin{equation}
I^{\PGSO(8) \rtimes \theta}_{\mathrm{disc},t,c}(f) = \sum_{c_2 \to c} \widehat{S}^{G_2}_{\mathrm{disc},t,c_2}(f') 
+ \frac{1}{4} \sum_{c_4 \to c} \widehat{S}^{\SO(4)}_{\mathrm{disc},t,c_4}(f'_4) 
+ \frac{1}{3} \sum_{c_3 \to c} \widehat{S}^{\SL_3}_{\mathrm{disc},t,c_3}(f'_3),
\end{equation}
where $c_2, c_4, c_3$ are summed over the sets of classes in $\mathcal{C}_{\mathbb{A}}(G_2)$, $\mathcal{C}_{\mathbb{A}}(\SO(4))$, $\mathcal{C}_{\mathbb{A}}(\SL(3))$ respectively that map to $c$.
\end{enumerate}
\end{lemma}

\begin{proof}
We prove that the stable linear form $S^{G_2}_{\mathrm{disc},t,c}$ exists, given that $S^{\SO(4)}_{\mathrm{disc},t,c}$ and $S^{\PGL(3)}_{\mathrm{disc},t,c}$ satisfy assertion (1).

For $f \in \mathcal{H}(G_2, K^{S})$ and $c^{S} \in \mathcal{C}_{\mathbb{A}}^{S}(G_2)$, set
\[
S^{G_2}_{\mathrm{disc},t,c^{S}}(f) = I^{G_2}_{\mathrm{disc},t,c^{S}}(f) 
- \iota(G,\SO(4)) \widehat{S}^{\SO(4)}_{\mathrm{disc},t,c^{S}}(f') 
- \iota(G,\PGL(3)) \widehat{S}^{\PGL(3)}_{\mathrm{disc},t,c^{S}}(f''),
\]
where
\[
\widehat{S}^{\SO(4)}_{\mathrm{disc},t,c^{S}}(f') = \sum_{c^{S}_4 \to c^{S}} \widehat{S}^{\SO(4)}_{\mathrm{disc},t,c_4^{S}}(f'), \qquad
\widehat{S}^{\PGL(3)}_{\mathrm{disc},t,c^{S}}(f'') = \sum_{c^{S}_3 \to c^{S}} \widehat{S}^{\PGL(3)}_{\mathrm{disc},t,c_3^{S}}(f'').
\]

If $h$ belongs to the unramified Hecke algebra $\mathcal{H}_{\mathrm{un}}^{S}(G_2)$, the fundamental lemma for spherical functions \cite{H} implies that
\[
(f_h)' = f'_{h'}, \qquad (f_h)'' = f''_{h''}.
\]

We then compute
\begin{align*}
S^{G_2}_{\mathrm{disc},t,c^{S}}(f_h)
&= I^{G_2}_{\mathrm{disc},t,c^{S}}(f_h) 
- \iota(G,\SO(4)) \widehat{S}^{\SO(4)}_{\mathrm{disc},t,c^{S}}((f_h)') 
- \iota(G,\PGL(3)) \widehat{S}^{\PGL(3)}_{\mathrm{disc},t,c^{S}}((f_h)'') \\
&= I^{G_2}_{\mathrm{disc},t,c^{S}}(f_h) 
- \iota(G,\SO(4)) \widehat{S}^{\SO(4)}_{\mathrm{disc},t,c^{S}}(f'_{h'}) 
- \iota(G,\PGL(3)) \widehat{S}^{\PGL(3)}_{\mathrm{disc},t,c^{S}}(f''_{h''}) \\
&= \hat{h}(c^{S}) I^{G_2}_{\mathrm{disc},t,c^{S}}(f) 
- \hat{h}(c^{S}) \bigl( \iota(G,\SO(4)) \widehat{S}^{\SO(4)}_{\mathrm{disc},t,c^{S}}(f') 
+ \iota(G,\PGL(3)) \widehat{S}^{\PGL(3)}_{\mathrm{disc},t,c^{S}}(f'') \bigr) \\
&= \hat{h}(c^{S}) S^{G_2}_{\mathrm{disc},t,c^{S}}(f),
\end{align*}
by the identity
\[
\hat{h}'(c^{S}_i) = \hat{h}(c^{S}).
\]
This is the analogue for $G_2$ of (\ref{eigenvector of invariant}).

We then obtain
\begin{align*}
\sum_{c^{S} \in \mathcal{C}_{\mathbb{A}}^{S}(G_2)} S^{G_2}_{\mathrm{disc},t,c^{S}}(f)
&= \sum_{c^{S}} I^{G_2}_{\mathrm{disc},t,c^{S}}(f) \\
&\quad - \sum_{c^{S}} \bigl( \iota(G,\SO(4)) \widehat{S}^{\SO(4)}_{\mathrm{disc},t,c^{S}}(f') 
+ \iota(G,\PGL(3)) \widehat{S}^{\PGL(3)}_{\mathrm{disc},t,c^{S}}(f'') \bigr) \\
&= I^{G_2}_{\mathrm{disc},t}(f) 
- \iota(G,\SO(4)) \sum_{c_4} \widehat{S}^{\SO(4)}_{\mathrm{disc},t,c_4}(f') 
- \iota(G,\PGL(3)) \sum_{c_3} \widehat{S}^{\PGL(3)}_{\mathrm{disc},t,c_3}(f'') \\
&= I^{G_2}_{\mathrm{disc},t}(f) 
- \iota(G,\SO(4)) \widehat{S}^{\SO(4)}_{\mathrm{disc},t}(f') 
- \iota(G,\PGL(3)) \widehat{S}^{\PGL(3)}_{\mathrm{disc},t}(f'') \\
&= S^{G_2}_{\mathrm{disc},t}(f).
\end{align*}

The sums over $c^{S}$, $c_4^{S}$ and $c_3^{S}$ can be taken over a finite set that depends on $f$ only through a choice of Hecke type. Since $I^{G_2}_{\mathrm{disc},t}(f)$ has the corresponding property, and the uniformity properties of the transfer mapping $f \mapsto f'$ follow from the two main theorems in \cite[\S 6]{A8}, the transfer map on the geometric side equals that on the spectral side. We set
\[
S^{G_2}_{\mathrm{disc},t,c}(f) = \sum_{c^{S} \to c} S^{G_2}_{\mathrm{disc},t,c^{S}}(f),
\]
which is a stable linear form. This completes the proof of assertion (1).

Now suppose $f$ lies in $\mathcal{H}(\PGSO(8) \rtimes \theta, K^{S})$. Consider the expression
\begin{equation} \label{c-component linear}
\sum_{c^{S}_2 \to c^{S}} \widehat{S}^{G_2}_{\mathrm{disc},t,c^{S}_2}(f') 
+ \frac{1}{4} \sum_{c^{S}_4 \to c^{S}} \widehat{S}^{\SO(4)}_{\mathrm{disc},t,c^{S}_4}(f'_4) 
+ \frac{1}{3} \sum_{c^{S}_3 \to c^{S}} \widehat{S}^{\SL_3}_{\mathrm{disc},t,c^{S}_3}(f'_3)
\end{equation}
attached to any $c^{S} \in \mathcal{C}_{\mathbb{A}}^{S}(\PGSO(8) \rtimes \theta)$. If $f$ is replaced by its transform $f_h$ under an element $h \in \mathcal{H}_{\mathrm{un}}^{S}$, then the expression is multiplied by the factor $\hat{h}(c^{S})$.

Summing (\ref{c-component linear}) over $c^{S}$ yields
\[
\widehat{S}^{G_2}_{\mathrm{disc},t}(f') 
+ \frac{1}{4} \widehat{S}^{\SO(4)}_{\mathrm{disc},t}(f'_4) 
+ \frac{1}{3} \widehat{S}^{\SL_3}_{\mathrm{disc},t}(f'_3) 
= I^{\PGSO(8) \rtimes \theta}_{\mathrm{disc},t}(f),
\]
by (\ref{c-component decomposition}). However, the linear form $I_{\mathrm{disc},t,c^{S}}(f)$ transforms under the action of $\mathcal{H}_{\mathrm{un}}^{S}(\PGSO(8) \rtimes \theta, K^{S})$, and its sum over $c^{S} \in \mathcal{C}_{\mathbb{A}}^{S}(\PGSO(8) \rtimes \theta)$ equals $I_{\mathrm{disc},t}(f)$. Thus, both $I_{\mathrm{disc},t,c^{S}}(f)$ and (\ref{c-component linear}) represent the $c^{S}$-component of $I_{\mathrm{disc},t}(f)$ relative to its decomposition into eigenfunctions under the action of $\mathcal{H}_{\mathrm{un}}^{S}(\PGSO(8) \rtimes \theta)$. Therefore (\ref{c-component linear}) equals $I_{\mathrm{disc},t,c^{S}}(f)$. Summing each side of this identity over those $c^{S}$ that map to a given class $c$ in $\mathcal{C}_{\mathbb{A}}(\PGSO(8) \rtimes \theta)$ yields the decomposition in (2).
\end{proof}

Any element $\psi \in \Psi^{\mathcal{O}}(\PGSO(8))$ gives rise to an archimedean infinitesimal character, the norm of whose imaginary part we denote by $t(\psi)$. There is a map $\psi \mapsto c(\psi)$ from $\Psi^{\mathcal{O}}(\PGSO(8))$ to $\mathcal{C}^{\mathcal{O}}(\PGSO(8))$ which is a bijection by Theorem \ref{classification of main theorem}. We set
\[
I_{\mathrm{disc},\psi}(\tilde{f}) = I^{\PGSO(8) \rtimes \theta}_{\mathrm{disc},t(\psi),c(\psi)}(\tilde{f}), \qquad \tilde{f} \in \mathcal{H}(\PGSO(8) \rtimes \theta),
\]
\[
I^{G_2}_{\mathrm{disc},\psi}(f) = I^{G_2}_{\mathrm{disc},t(\psi),c(\psi)}(f),
\]
and
\[
S^{G_2}_{\mathrm{disc},\psi}(f) = S^{G_2}_{\mathrm{disc},t(\psi),c(\psi)}(f),
\]
for any $\psi \in \Psi^{\mathcal{O}}(\PGSO(8))$ and $f \in \mathcal{H}(G_2)$.

\begin{corollary}
Suppose that $\psi \in \Psi^{\mathcal{O}}(\PGSO(8))$.

\begin{enumerate}
\item[(1)] If $\tilde{f}$ belongs to $\mathcal{H}(\PGSO(8) \rtimes \theta)$, we have
\[
I_{\mathrm{disc},\psi}(\tilde{f}) = \widehat{S}^{G_2}_{\mathrm{disc},\psi}(f') 
+ \frac{1}{4} \widehat{S}^{\SO(4)}_{\mathrm{disc},\psi}(f'_4) 
+ \frac{1}{3} \widehat{S}^{\SL_3}_{\mathrm{disc},\psi}(f'_3).
\]

\item[(2)] If $f$ belongs to $\mathcal{H}(G_2)$, we have
\begin{align*}
I^{G_2}_{\mathrm{disc},\psi}(f)
&= S^{G_2}_{\mathrm{disc},\psi}(f) 
+ \iota(G,\SO(4)) \widehat{S}^{\SO(4)}_{\mathrm{disc},\psi}(f') 
+ \iota(G,\PGL(3)) \widehat{S}^{\PGL(3)}_{\mathrm{disc},\psi}(f'') \\
&= \operatorname{tr}\bigl( \mathcal{I}_{G,\psi}(1,f) \bigr) \\
&\quad + \frac{1}{6} \sum_{w \in W(\GL(2)_s)_{\mathrm{reg}}} |\det(w-1)_{\mathbb{R}}|^{-1} 
\operatorname{tr}\bigl( M_{P,\psi}(w,1) \mathcal{I}_{P,\psi}(1,f) \bigr) \\
&\quad + \frac{1}{6} \sum_{w \in W(\GL(2)_l)_{\mathrm{reg}}} |\det(w-1)_{\mathbb{R}}|^{-1} 
\operatorname{tr}\bigl( M_{P',\psi}(w,1) \mathcal{I}_{P',\psi}(1,f) \bigr) \\
&\quad + \frac{1}{12} \sum_{w \in W(T)_{\mathrm{reg}}} |\det(w-1)_{\mathbb{R}^2}|^{-1} 
\operatorname{tr}\bigl( M_{B,\psi}(w,\chi) \mathcal{I}_{B,\psi}(1,f) \bigr).
\end{align*}
\end{enumerate}
\end{corollary}
  
\section{A coarse classification of automorphic representation of $\G_{2}$}

 We have a bijection $\psi \mapsto c(\psi)$ from $\Psi^{\mathcal{O}}(\PGSO(8))$ onto $\mathcal{C}^{\mathcal{O}}(\PGSO(8))$. Here $\mathcal{C}^{\mathcal{O}}(\PGSO(8))$ denotes the set of $A$-packets that contain the $\mathcal{O}$-stable representations occurring in the automorphic spectrum of $\PGSO(8)$. These $A$-packets contain all representations that appear in the formula (\ref{invariant trace formula of PGSO}) for $\PGSO(8) \rtimes \theta$. Consequently, the summand indexed by $c$ in (\ref{c-component decomposition}) vanishes unless $c = c(\psi)$ for some $\psi \in \Psi^{\mathcal{O}}(\PGSO(8))$. We thus obtain the decomposition
\begin{equation}
I^{\PGSO(8) \rtimes \theta}_{\mathrm{disc},t}(\tilde{f}) = \sum_{\{ \psi \in \Psi^{\mathcal{O}}(\PGSO(8)) : t(\psi) = t \}} I_{\mathrm{disc},\psi}(\tilde{f}), \qquad \tilde{f} \in \mathcal{H}^{\mathcal{O}}(\PGSO(8)).
\end{equation}

The following theorem establishes an analogous formula for $G_2$. It reduces the study of the automorphic spectrum of $G_2$ to the subset of parameters $\psi \in \Psi^{\mathcal{O}}(\PGSO(8))$.

\begin{theorem} \label{main theorem 1}
Suppose $f \in \mathcal{H}(G_2)$, $t \geq 0$, and $c \in \mathcal{C}_{\mathbb{A}}(G_2)$. Then
\[
I^{G_2}_{\mathrm{disc},t,c}(f) = 0 = S^{G_2}_{\mathrm{disc},t,c}(f),
\]
unless
\[
(t,c) = \bigl( t(\psi), c(\psi) \bigr)
\]
for some $\psi \in \Psi^{\mathcal{O}}(\PGSO(8))$.
\end{theorem}

\begin{proof}
Assume that $(t,c) \neq (t(\psi), c(\psi))$ for any $\psi \in \Psi^{\mathcal{O}}(\PGSO(8))$. Then
\[
I_{\mathrm{disc},t,c}^{\PGSO(8)}(\tilde{f}) = 0.
\]

Applying Lemma \ref{stable distribution exist} (2), we have
\[
I^{\PGSO(8) \rtimes \theta}_{\mathrm{disc},t,c}(\tilde{f}) = 
\sum_{c_2 \to c} \widehat{S}^{G_2}_{\mathrm{disc},t,c_2}(f') 
+ \frac{1}{4} \sum_{c_4 \to c} \widehat{S}^{\SO(4)}_{\mathrm{disc},t,c_4}(f'_4) 
+ \frac{1}{3} \sum_{c_3 \to c} \widehat{S}^{\SL_3}_{\mathrm{disc},t,c_3}(f'_3).
\]

Since $(t,c) \neq (t(\psi), c(\psi))$, Arthur's work \cite{A1} implies
\[
(t,c_4) \neq \bigl( t(\psi_4), c(\psi_4) \bigr) \quad \text{for any } \psi_4 \in \widetilde{\Psi}(\SO(4)), 
\]
\[
(t,c_3) \neq \bigl( t(\psi_3), c(\psi_3) \bigr) \quad \text{for any } \psi_3 \in \Psi(\SL(3)).
\]
Hence
\[
\widehat{S}^{\SO(4)}_{\mathrm{disc},t,c_4}(f'_4) = 0 = \widehat{S}^{\SL(3)}_{\mathrm{disc},t,c_3}(f'_3).
\]

Therefore
\[
S^{G_2}_{\mathrm{disc},t,c}(f) = \widehat{S}^{G_2}_{\mathrm{disc},t,c}(f') = 
\sum_{c_2 \to c} \widehat{S}^{G_2}_{\mathrm{disc},t,c_2}(f') = 0.
\]

Applying Lemma \ref{stable distribution exist} (1), we have the decomposition
\[
I^{G_2}_{\mathrm{disc},t,c}(f) = S^{G_2}_{\mathrm{disc},t,c}(f) 
+ \iota(G,\SO(4)) \widehat{S}^{\SO(4)}_{\mathrm{disc},t,c}(f') 
+ \iota(G,\PGL(3)) \widehat{S}^{\PGL(3)}_{\mathrm{disc},t,c}(f'').
\]
Since
\[
\widehat{S}^{\SO(4)}_{\mathrm{disc},t,c}(f') = 0 = \widehat{S}^{\PGL(3)}_{\mathrm{disc},t,c}(f''),
\]
and $S^{G_2}_{\mathrm{disc},t,c} = 0$, we conclude
\[
I^{G_2}_{\mathrm{disc},t,c}(f) = 0.
\]
\end{proof}

The following corollary follows immediately from the decompositions (\ref{c-compoent expand}) and (\ref{Stable of G}).

\begin{corollary}
For any $t \geq 0$ and $f \in \mathcal{H}(G_2)$, we have
\[
I^{G_2}_{\mathrm{disc},t}(f) = \sum_{\{ \psi \in \Psi^{\mathcal{O}}(\PGSO(8)) : t(\psi) = t \}} I^{G_2}_{\mathrm{disc},\psi}(f)
\]
and
\[
S^{G_2}_{\mathrm{disc},t}(f) = \sum_{\{ \psi \in \Psi^{\mathcal{O}}(\PGSO(8)) : t(\psi) = t \}} S^{G_2}_{\mathrm{disc},\psi}(f).
\]
\end{corollary}

We define an invariant subspace
\[
L^{2}_{\mathrm{disc},t,c}\bigl( G_2(F) \backslash G_2(\mathbb{A}) \bigr), \qquad t \geq 0,\; c \in \mathcal{C}_{\mathbb{A}}(G_2),
\]
of $L^{2}_{\mathrm{disc}}(G_2(F) \backslash G_2(\mathbb{A}))$ as the direct sum
\[
\bigoplus_{\{ \pi : \|\mu_{\pi,I}\| = t,\; c(\pi) = c \}} m(\pi) \, \pi,
\]
of irreducible representations $\pi$ of $G_2(\mathbb{A})$ attached to $t$ and $c$ that occur in $L^{2}_{\mathrm{disc}}(G_2(F) \backslash G_2(\mathbb{A}))$ with positive multiplicity $m(\pi)$. Then
\[
L^{2}_{\mathrm{disc}}\bigl( G_2(F) \backslash G_2(\mathbb{A}) \bigr) = 
\bigoplus_{t \geq 0,\; c \in \mathcal{C}_{\mathbb{A}}(G_2)} L^{2}_{\mathrm{disc},t,c}\bigl( G_2(F) \backslash G_2(\mathbb{A}) \bigr).
\]

If $\psi \in \Psi^{\mathcal{O}}(\PGSO(8))$, we denote
\[
L^{2}_{\mathrm{disc},\psi}\bigl( G_2(F) \backslash G_2(\mathbb{A}) \bigr) = 
L^{2}_{\mathrm{disc},t(\psi),c(\psi)}\bigl( G_2(F) \backslash G_2(\mathbb{A}) \bigr).
\]

\begin{theorem}
If $t \geq 0$ and $c \in \mathcal{C}_{\mathbb{A}}(G_2)$, then
\[
L^{2}_{\mathrm{disc},t,c}\bigl( G_2(F) \backslash G_2(\mathbb{A}) \bigr) = 0,
\]
unless
\[
(t,c) = (t(\psi), c(\psi))
\]
for some $\psi \in \Psi^{\mathcal{O}}(\PGSO(8))$. In particular, we have a decomposition
\[
L^{2}_{\mathrm{disc}}\bigl( G_2(F) \backslash G_2(\mathbb{A}) \bigr) = 
\bigoplus_{\psi \in \Psi^{\mathcal{O}}(\PGSO(8))} L^{2}_{\mathrm{disc},\psi}\bigl( G_2(F) \backslash G_2(\mathbb{A}) \bigr).
\]
\end{theorem}

\begin{proof}
Assume that $(t,c)$ is not of the form $(t(\psi), c(\psi))$. Theorem \ref{main theorem 1} asserts that $I^{G_2}_{\mathrm{disc},t,c}(f)$ vanishes. We have the decomposition
\begin{align*}
I^{G_2}_{\mathrm{disc},t,c}(f)
&= \operatorname{tr}\bigl( \mathcal{I}_{G_2,t,c}(1,f) \bigr) \\
&\quad + \frac{1}{6} \sum_{w \in W(\GL(2)_s)_{\mathrm{reg}}} |\det(w-1)_{\mathbb{R}}|^{-1} 
\operatorname{tr}\bigl( M_{P,t,c}(w,1) \mathcal{I}_{P,t,c}(1,f) \bigr) \\
&\quad + \frac{1}{6} \sum_{w \in W(\GL(2)_l)_{\mathrm{reg}}} |\det(w-1)_{\mathbb{R}}|^{-1} 
\operatorname{tr}\bigl( M_{P',t,c}(w,1) \mathcal{I}_{P',t,c}(1,f) \bigr) \\
&\quad + \frac{1}{12} \sum_{w \in W(T)_{\mathrm{reg}}} |\det(w-1)_{\mathbb{R}^2}|^{-1} 
\operatorname{tr}\bigl( M_{B,t,c}(w,\chi) \mathcal{I}_{B,t,c}(1,f) \bigr).
\end{align*}

Let $M$ denote a proper Levi subgroup of $G_2$, i.e., $M = \GL(2)_s$, $\GL(2)_l$, or $T$. The operator $I^{G_2}_{P,t,c}(1_M, f)$ is induced from the $(t,c)$-component of the automorphic discrete spectrum of $M$. By Arthur's work \cite[$\S$1]{A1}, this operator vanishes, and hence so does the corresponding summand above.

Therefore we obtain
\[
I^{G_2}_{\mathrm{disc},t,c}(f) = \operatorname{tr}\bigl( \mathcal{I}_{G_2,t,c}(1_{G_2}, f) \bigr) = 0.
\]

Since
\[
\mathcal{I}_{G_2,t,c}(f) = \mathcal{I}_{G_2,t,c}(1_{G_2}, f)
\]
is the operator on $L^{2}_{\mathrm{disc},t,c}(G_2(F) \backslash G_2(\mathbb{A}))$ given by right convolution by $f$, its trace equals the sum
\[
\sum_{\pi} m(\pi) \operatorname{tr}\bigl( \pi(f) \bigr), \qquad
\|\mu_{\pi,I}\| = t,\; c(\pi) = c,
\]
over irreducible subrepresentations of this space. If all multiplicities $m(\pi)$ were positive, this sum could not vanish for every $f \in \mathcal{H}(G_2)$. Hence $m(\pi) = 0$ for each $\pi$, and the sum is over an empty set. Consequently,
\[
L^{2}_{\mathrm{disc},t,c}\bigl( G_2(F) \backslash G_2(\mathbb{A}) \bigr) = 0.
\]
\end{proof}

\begin{remark}
We have used the $A$-parameters of $\PGSO(8)$ to classify the automorphic representations of $G_2$. For some $\psi \in \Psi^{\mathcal{O}}(\PGSO(8))$, the space $L^{2}_{\mathrm{disc},\psi}(G_2(F) \backslash G_2(\mathbb{A}))$ may vanish. It remains to determine which $A$-parameters $\psi$ contribute to the discrete spectrum of $G_2$ and to give an explicit stable multiplicity formula. In a subsequent paper, we will study the stable distribution $S^{G_2}_{\mathrm{disc},\psi}(f)$ in detail.
\end{remark}

\end{document}